\documentclass[12pt]{amsart}

\usepackage{amsmath}
\usepackage[all]{xy}
\usepackage[margin=1.00in]{geometry}
\usepackage{enumitem}
\usepackage{setspace}
\usepackage{nicefrac}
\usepackage{xspace}
\usepackage{tikz}
\usetikzlibrary{calc}

\usepackage{microtype} 

\usepackage{amsthm}
\usepackage{thmtools}
\declaretheorem[numberwithin=section]{theorem}

\declaretheorem[sibling=theorem]{lemma}
\declaretheorem[sibling=theorem]{proposition}
\declaretheorem{question}
\declaretheorem[style=definition]{definition}
\declaretheorem[style=remark,numbered=no]{remark}
\declaretheorem[name=Some remarks,style=remark,numbered=no]{remarks}
\declaretheorem[style=definition,numberwithin=section]{example}

\newcommand{\rat}{\dashrightarrow}
\newcommand{\set}[1]{\left\{ #1 \right\}}
\newcommand{\abs}[1]{\left\lvert #1 \right\rvert}
\newcommand{\ang}[1]{\left\langle #1 \right\rangle}
\newcommand{\red}{\text{red}}

\newcommand{\cI}{\mathcal I}
\newcommand{\wR}{\widehat{R}}

\newcommand{\fm}{\mathfrak m}
\newcommand{\fn}{\mathfrak n}

\newcommand{\nth}{$^\text{\tiny th}$\xspace}

\DeclareMathOperator{\Alb}{Alb}
\DeclareMathOperator{\Aut}{Aut}
\DeclareMathOperator{\Bl}{Bl}

\DeclareMathOperator{\SL}{SL}
\DeclareMathOperator{\id}{id}

\DeclareMathOperator{\Sing}{Sing}

\DeclareMathOperator{\Stab}{Stab}

\DeclareMathOperator{\Chow}{Chow}
\DeclareMathOperator{\Hilb}{Hilb}
\DeclareMathOperator{\Pic}{Pic}
\DeclareMathOperator{\Proj}{Proj}
\DeclareMathOperator{\Univ}{Univ}

\DeclareMathOperator{\NEb}{\overline{NE}}
\DeclareMathOperator{\Nef}{Nef}

\renewcommand{\P}{\mathbb P}

\newcommand{\A}{\mathbb A}
\newcommand{\C}{\mathbb C}
\newcommand{\F}{\mathbb F}
\newcommand{\Q}{\mathbb Q}
\newcommand{\R}{\mathbb R}
\newcommand{\Z}{\mathbb Z}
\newcommand{\cE}{\mathcal E}
\newcommand{\cO}{\mathcal O}

\newcommand{\Ra}{\mathbb R_{\geq 0} \, }
\newcommand{\Ras}{\mathbb R_{> 0} \,}

\title[Automorphisms of smooth threefolds]{Some constraints on positive entropy \\ automorphisms of smooth threefolds}
\author{John Lesieutre}
\address{Institute for Advanced Study, Einstein Drive, Princeton, NJ 08540 USA}
\email{johnl@math.ias.edu}

\subjclass[2010]{14J50,14E07,37F99,14E30}

\begin{document}

\begin{abstract}
Suppose that \(X\) is a smooth, projective threefold over \(\C\) and that \(\phi : X \to X\) is an automorphism of positive entropy.  We show that one of the following must hold, after replacing \(\phi\) by an iterate: i) the canonical class of \(X\) is numerically trivial; ii) \(\phi\) is imprimitive; iii) \(\phi\) is not dynamically minimal.  As a consequence, we show that if a smooth threefold \(M\) does not admit a primitive automorphism of positive entropy, then no variety constructed by a sequence of smooth blow-ups of \(M\) can admit a primitive automorphism of positive entropy. 

In explaining why the method does not apply to threefolds with terminal singularities, we exhibit a non-uniruled, terminal threefold \(X\) with infinitely many \(K_X\)-negative extremal rays on \(\NEb(X)\).
\end{abstract}

\maketitle

\section{Introduction}

Suppose that \(X\) is a smooth projective variety over \(\C\).  An automorphism \(\phi : X \to X\) is said to have \emph{positive entropy} if the pullback map \(\phi^\ast : N^1(X) \to N^1(X)\) has an eigenvalue greater than \(1\).  By a fundamental result of Gromov and Yomdin, this notion of positive entropy coincides with the one familiar in dynamical systems, related to the separation of orbits by \(\phi\); we refer to \cite{oguisoicm} for an excellent survey of these results.

Although there are many interesting examples of positive entropy automorphisms of projective surfaces, examples in higher dimensions remain scarce.  Our aim in this note is to give some constraints on the geometry of smooth, projective threefolds that admit automorphisms of positive entropy and partly explain this scarcity. These constraints are specific to automorphisms of threefolds: they hold neither for automorphisms of surfaces, nor for pseudoautomorphisms of threefolds.

Before stating the main results, we recall two basic ways in which an automorphism of \(X\) can be built out of automorphisms of ``simpler'' varieties.

\begin{definition}
\label{primitivedef}
An automorphism \(\phi : X \to X\) is \emph{imprimitive} if there  exists a variety \(V\) with   \(1 \leq \dim V < \dim X\), a birational automorphism \(\psi : V \rat V\), and a dominant rational map \(\pi : X \rat V\) such that \(\pi \circ \phi = \psi \circ \pi\).    The map \(\phi\) is called \emph{primitive} if it is not imprimitive~\cite{dqzhangmmp}.  
\end{definition}
For example, if \(\psi : V \to V\) is a positive entropy automorphism, the induced map \(\phi : \P(TV) \to \P(TV)\) of the projectivized tangent bundle also has positive entropy, but is not primitive.

\begin{definition}
\label{minimaldef}
An automorphism \(\phi : X \to X\) is \emph{not dynamically minimal} if there exists a variety \(Y\) with terminal singularities, a birational morphism \(\pi : X \to Y\), and an automorphism \(\psi : Y \to Y\) with \(\pi \circ \phi = \psi \circ \pi\).  If no such \(\pi : X \to Y\) exists, \(\phi\) is called \emph{dynamically minimal}. 
\end{definition}
For example, if \(\psi : Y \to Y\) is a positive entropy automorphism, and \(V \subset Y\) is a \(\psi\)-invariant subvariety, there is an induced automorphism \(\phi : \Bl_V Y \to \Bl_V Y\).  The map \(\phi\) has positive entropy, but it is not dynamically minimal.

The restriction that \(Y\) have terminal singularities is quite natural from the point of view of birational geometry, for these are the singularities that can arise in running the minimal model program (MMP) on \(X\). In dimension \(2\), having terminal singularities is equivalent to smoothness, and dynamical minimality is equivalent to the non-existence of \(\phi\)-periodic \((-1)\)-curves on \(X\).

Positive entropy automorphisms of projective surfaces are in many respects well-understood.  If \(X\) is a smooth projective surface admitting a positive entropy automorphism, it must be a blow-up of either \(\P^2\), a K3 surface, an abelian surface, or an Enriques surface~\cite{cantatbasic}.  Blow-ups of \(\P^2\) at \(10\) or more points have proved to be an especially fertile source of examples, beginning with work of Bedford and Kim~\cite{bedfordkim} and McMullen~\cite{mcmullen}.    However, in higher dimensions, there are very few examples known of primitive, positive entropy automorphisms. The first such example on a smooth, rational threefold was given only recently by Oguiso and Truong~\cite{oguisotruong}.  

One result of this note is that the three-dimensional analogs of the basic blow-up constructions in dimension two can never yield primitive, positive entropy automorphisms.

\begin{restatable}{theorem}{thmblowupautos}
\label{blowupautos}
Suppose that \(M\) is a smooth projective threefold that does not admit any automorphism of positive entropy, and that \(X\) is constructed by a sequence of blow-ups of \(M\) along smooth centers. Then any positive entropy automorphism \(\phi : X \to X\) is imprimitive.
\end{restatable}

This provides a partial answer to a question of Bedford:
\begin{question}[Bedford, cf.\ \cite{truong}]
\label{bedfordquestion}
Does there exist a smooth blow-up of \(\P^3\) admitting a positive entropy automorphism?
\end{question}

According to Theorem~\ref{blowupautos}, if such an automorphism exists, it must be imprimitive.  Truong has also obtained many results on this question, showing that if \(X\) is constructed by a sequence of blow-ups of points and curves whose normal bundles satisfy certain constraints, then \(X\) admits no positive entropy automorphisms, and that under certain weaker conditions, any positive entropy automorphism has equal first and second dynamical degrees~\cite{truong},\cite{moretruong}.

Theorem~\ref{blowupautos} applies to only a fairly narrow class of threefolds: whereas every smooth projective surface can be obtained as the blow-up of a minimal surface, this is far from true for threefolds.  Although the sharpest results we obtain are in this blow-up setting, in combination with classification results from the MMP, the same approach yields some constraints on positive entropy automorphisms of arbitrary smooth threefolds for which the canonical class is not numerically trivial.

\begin{restatable}{theorem}{thmmain}
\label{main}
Suppose that \(X\) is a smooth projective threefold and that \(\phi : X \to X\) is an automorphism of positive entropy.  After replacing \(\phi\) by some iterate, at least one of the following must hold:
\begin{enumerate}
\item the canonical class of \(X\) is numerically trivial;
\item \(\phi\) is imprimitive;
\item \(\phi\) is not dynamically minimal.
\end{enumerate}
\end{restatable}
The conclusions in all these cases can be refined considerably; a more detailed statement appears as Theorem~\ref{precisemain} below.  In Section~\ref{examplelist}, we catalog some typical instances of each case.  

We caution that Theorem~\ref{main} should not be construed as a classification of threefolds admitting a primitive automorphism of positive entropy.  The chief difficulty lies in case (3): when \(\phi\) is not dynamically minimal, the new variety \(Y\) on which \(\phi\) induces an automorphism may no longer be smooth, so the result can not be applied inductively.  This leads to the following.

\begin{restatable}{corollary}{thminduction}
\label{induction}
Suppose that \(\phi : X \to X\) is a primitive, positive entropy automorphism of a smooth, projective, rationally connected threefold.  Then there exists a non-smooth threefold \(Y\) with terminal singularities and a birational map \(\pi : X \to Y\) such that some iterate of \(\phi\) descends to an automorphism of \(Y\).
\end{restatable}

Experience with the MMP suggests that it is unsurprising that even in studying automorphisms of smooth threefolds, it is useful to consider threefolds with terminal singularities.  The unexpected feature of Corollary~\ref{induction} is that such singularities on \(Y\) are not only allowed, but unavoidable.

Results of Zhang show that if \(X\) admits a primitive, positive entropy automorphism, it must be either rationally connected or birational to a variety with numerically trivial canonical class~\cite{dqzhangmmp}.  Our results are primarily of interest in the rationally connected setting, and are in some sense complementary to those of Zhang: although we obtain no new information on the birational type of \(X\), we give some constraints on the geometry of a birational model on which \(\phi\) acts as an automorphism.  For example, we obtain the following corollary, which was previously shown by Bedford and Kim to be false in dimension \(2\) (see Example~\ref{bedfordkimnocurves}).

\begin{restatable}{corollary}{thminvariantstuff}
\label{invariantstuff}
Suppose that \(\phi : X \to X\) is a primitive, positive entropy automorphism of a smooth projective threefold.  If \(K_X\) is not numerically trivial, then there exists a \(\phi\)-invariant divisor on \(X\).
\end{restatable}

A shortcoming of Corollary~\ref{induction} is that it fails to give any further information about the automorphism of the singular model \(Y\).  The essential problem is that running the MMP on \(Y\) might require performing a flip \(Y \rat Y^+\).  If the flipping curve has infinite orbit under \(\phi\), the induced map on \(Y^+\) will be only a pseudoautomorphism.  In Section~\ref{singularcase}, we illustrate the difficulty in the case of an example of Oguiso and Truong.  The flipping curve in this instance has infinite orbit under \(\phi\), and our methods do not apply.  However, passing to a suitable branched cover, we obtain:
\begin{restatable*}{theorem}{thmmanyflips}
\label{manyflips}
There exists a terminal, projective threefold \(Y\) of non-negative Kodaira dimension with infinitely many \(K_Y\)-negative extremal rays on \(\NEb(Y)\).
\end{restatable*}
This provides a new negative answer to a question of Kawamata, Matsuda, and Matsuki:
\begin{question}[{\cite[Problem 4-2-5]{kmm}}]
\label{kmmquestion}
Suppose that \(X\) is a terminal variety.  According to the cone theorem,
\[
\NEb(X) = \NEb_{K_X \geq 0}(X) + \sum_i \Ra [C_i].
\]
If \(\kappa(X) \geq 0\), must the number of \(K_X\)-negative extremal rays be finite?
\end{question}
The standard example of a variety with infinitely many \(K_X\)-negative extremal rays is the blow-up of \(\P^2\) at \(9\) or more very general points, when the infinitely many \((-1)\)-curves on \(X\) generate such rays.  The restriction that \(\kappa(X) \geq 0\) excludes any simple variations on this example.  A negative answer to Question~\ref{kmmquestion} where \((X,\Delta)\) is a klt pair was noted by Uehara, answering \cite[Problem 4-2-5]{kmm} in its original formulation~\cite{uehara}.  The example of Theorem~\ref{manyflips} seems to be the first with \(\Delta = 0\) (cf.\ \cite{uehara}, \cite[Remark III.1.2.5.1]{kollarrationalcurves}).

The proof of Theorem~\ref{main} makes use of the interplay between the global geometry of \(X\), as governed by the MMP, and the local dynamics of the automorphism \(\phi : X \to X\) around certain invariant subvarieties.  A key technical input is the following observation about the local dynamics of an automorphism of a threefold around an invariant curve, roughly extending a result of Arnold from the surface case.  We hope that this may be of some independent interest.

\begin{restatable*}{theorem}{thmseparateiterates}
\label{separateiterates}
Suppose that \(X\) is a smooth, projective threefold with an infinite order automorphism \(\phi : X \to X\).  Let \(C\) be an irreducible curve with \(\phi(C) = C\).  Suppose that \(E \subset X\) is an irreducible divisor, containing \(C\) and nonsingular at the generic point of \(C\), and which is not \(\phi\)-periodic.  Then there exists a smooth, projective threefold \(Y\) with a  birational morphism \(\pi : Y \to X\) such that, after replacing \(\phi\) by an iterate:
\begin{enumerate}
\item The map \(\phi\) lifts to an automorphism of \(Y\);
\item \(\pi : Y \setminus \pi^{-1}(C) \to X \setminus C\) is an isomorphism;
\item \(\pi(E_m \cap E_n)\) does not contain \(C\) for any \(m \neq n\).
\end{enumerate}
\end{restatable*}

The following theorem provides a more detailed breakdown of the various subcases in the statement of Theorem~\ref{main}. 

\begin{theorem}
\label{precisemain}
Suppose that \(X\) is a smooth projective threefold and that \(\phi : X \to X\) is an automorphism of positive entropy.  After replacing \(\phi\) by some iterate, at least one of the following must hold:
\begin{enumerate}
\item the canonical class \(K_X\) is numerically trivial and either:
\begin{enumerate}
 \item \(X\) is an abelian threefold;
 \item \(X\) is a weak Calabi-Yau variety: \(K_X\) is torsion in \(\Pic(X)\) and \(h^{0,1}(X) = 0\);
\end{enumerate}
\item \(\phi\) is imprimitive and either:
\begin{enumerate}
\item the canonical class \(K_X\) is semiample and the canonical fibration \(\pi : X \to X_{\mathrm{can}}\) realizes \(\phi\) as imprimitive;
\item there exists a conic bundle \(\pi : X \to V\) with  \(\rho(X/V) = 1\) realizing \(\phi\) as imprimitive;
\item there exists a surface \(S\) with an automorphism \(\psi : S \to S\), a birational morphism \(\pi : X^\prime \to X\) such that \(\phi\) lifts to an automorphism \(\bar{\phi} : X^\prime \to X^\prime\), and a morphism \(\rho : X^\prime \to S\) with all fibers \(1\)-dimensional such that \(\rho \circ \bar{\phi} = \psi \circ \rho\).
\end{enumerate}
\item \(\phi\) is not dynamically minimal: there exists a divisorial contraction \(\pi : X \to Y\), where \(Y\) has terminal singularities, and \(\phi\) descends to an automorphism \(\psi : Y \to Y\).  Either:
\begin{enumerate}
\item \(Y\) is smooth;
\item the unique singularity of \(Y\) is locally analytically isomorphic to \(w^2+x^2+y^2+z^2 = 0\), \(w^2+x^2+y^2+z^3 = 0\), or the cone over the Veronese surface in \(\P^5\).
\end{enumerate}
\end{enumerate}
\end{theorem}

We next outline the strategy of the proof of Theorem~\ref{precisemain}.  There are three main steps.

\medskip
\noindent\emph{Step 1: Initial reductions from the MMP}

First we show that it is possible to make several simplifying assumptions on \(X\). If \(K_X\) is nef, the arguments of Zhang show that \(X\) satisfies one of Case (1) or 2(a).   If \(K_X\) is not nef, we consider the first step of the MMP for \(X\). There exists a contraction of a \(K_X\)-negative extremal ray, and the proof breaks into three cases:
\begin{enumerate}
\item there is a Mori fiber space \(\pi : X \to Y\);
\item there is a divisorial contraction \(\pi : X \to Y\), and the exceptional divisor \(E\) is \(\phi\)-periodic;
\item there is a divisorial contraction \(\pi : X \to Y\), and the exceptional divisor \(E\) is not \(\phi\)-periodic.
\end{enumerate}
Since \(X\) is a smooth threefold, there are no flipping contractions.  If \(X\) is constructed as a smooth blow-up, as in Theorem~\ref{blowupautos}, we may assume that we are in Case (2) or (3).  The divisorial contraction \(\pi : X \to Y\) is just the final blow-up map, with exceptional divisor \(E\). 

In Case (1), it follows from a lemma of Wi\'sniewski that some iterate of \(\phi\) is imprimitive.  In Case (2), we replace \(\phi\) by an iterate \(\phi^n\) fixing \(E\).  Perhaps after once more replacing \(\phi\) with \(\phi^2\) (to handle the case that \(E \cong \P^1 \times \P^1\) and \(\phi\) exchanges the rulings), \(\phi\) then descends to an automorphism of \(Y\).  This shows that \(\phi\) is not dynamically minimal.

The bulk of the work is in the remaining Case (3): we must show that if the exceptional divisor \(E\) is not \(\phi\)-periodic, then \(\phi\) is necessarily imprimitive. The argument hinges on the fact that the exceptional divisor \(E\) must be a smooth ruled surface (i.e.\ a \(\P^1\)-bundle over a curve), and the geometry of such surfaces is fairly simple. The presence of an infinite set of contractible, ruled surfaces \(\phi^n(E) \subset X\) has strong geometric implications.

\medskip
\noindent\emph{Step 2: Numerical consequences of positive entropy}

The next step is to translate the condition of positive entropy into a form that can be used to give geometric conclusions.  The starting point is an observation of Truong~\cite{truong}: there is (perhaps after replacing \(\phi\) with \(\phi^{-1}\)) a dominant eigenvector \(D\) of the pull-back \(\phi^\ast : N^1(X) \to N^1(X)\), such that \(D\) is a nef divisor with \(D^2 = 0\).  We will exploit the properties of this divisor to control the intersections \(\phi^m(E) \cap \phi^n(E)\) and ultimately show that \(\phi\) is imprimitive.

The exceptional divisor of the contraction \(\pi : X \to Y\) is a smooth ruled surface \(E \subset X\).  The restriction \(D \vert_E\) is a nef divisor with \((D\vert_E \cdot D\vert_E)_E = 0\); thus \(D\vert_E\) is not ample, and it lies on the boundary of the nef cone \(\Nef(E)\).  Because a ruled surface has Picard rank \(2\), there are only three such divisors, up to rescaling: the zero divisor, the class of a fiber, and a second boundary ray.  Moreover, after an appropriate rescaling, we can assume that \(D \vert_E\) is actually a rational class, even though \(D\) itself is not.  The cases in which \(D \vert_E\) is zero or equivalent to a fiber can be quickly excluded, so we may assume that \(D \vert_E\) is on the second boundary ray of \(\Nef(E)\).  An intersection-theoretic trick then shows that for all nonzero \(n\), the divisors \(\phi^n(E) \vert_E\) have numerical class proportional to \(\alpha\), a generator of one of the two extremal rays on the cone of curves \(\NEb(E)\).

\medskip
\noindent\emph{Step 3: From numerical data to an equivariant
fibration}

The geometry of the ruled surface \(E\) now enables us to draw some geometric conclusions, using the condition that \([\phi^n(E) \vert_E]\) is extremal on \(\NEb(E)\).  There are two possibilities, depending on whether the set of curves in \(E\) with numerical class on the extremal ray \(\Ras \alpha\) is finite or infinite. Both situations are possible: for example, if \(E \cong \F_n\) is a Hirzebruch surface with \(n \geq 1\), then \(\Ras \alpha\) is represented only by the negative section, while if \(E \cong \P^1 \times \P^1\), then \(\Ras \alpha\) is represented by a \(1\)-dimensional family of sections.

Suppose first that \(\Ras \alpha\) is represented by a one-dimensional algebraic family of curves, and that \(\phi^n(E) \cap E\) contains infinitely many different curves inside \(E\) as \(n\) varies.  In this case, we show that there exists a curve \(\xi \subset E\) that moves in algebraic families covering \(\phi^n(E)\) for infinitely many different values of \(n\).  A Hilbert scheme argument implies that \(\xi\) must in fact deform in a family of dimension at least \(2\), covering all of \(X\).   The map \(\phi\) sends curves in the deformation class of \(\xi\) to other curves in the deformation class of \(\xi\), and so \(\phi\) induces an automorphism of the space parametrizing such curves: this parameter space is a an irreducible component \(\Hilb_{[\xi]}(X)\) of the Hilbert scheme \(\Hilb(X)\).  

We next argue that \(\Hilb_{[\xi]}(X)\) is two-dimensional and that there exists a \(\phi\)-equivariant rational map \(X \rat \Hilb_{[\xi]}(X)\).  The essential point is that deformations of \(\xi\) ``exactly'' cover \(X\), in the sense that through a general point \(x\) on \(X\) there is a unique curve \(\xi^\prime\) deformation-equivalent to \(\xi\). The map \(X \rat \Hilb_{[\xi]}(X)\) then sends \(x\) to the point \([\xi^\prime]\) on the Hilbert scheme parametrizing this curve.  In other words, the corresponding component \(\Univ_{[\xi]}(X)\) of the universal family maps birationally to \(X\), and the composition \(X \rat \Univ_{[\xi]}(X) \to \Hilb_{[\xi]}(X)\)  presents \(\phi\) as an imprimitive map.

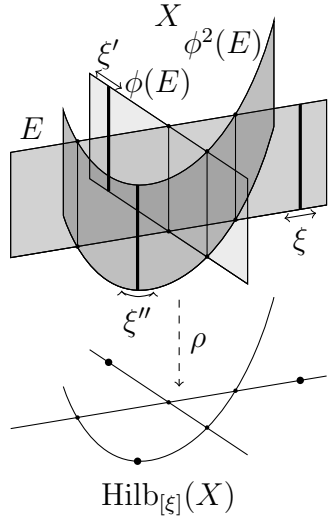
\begin{figure}[htb]
  \centering
\begin{tikzpicture}[scale=0.7]
\def\yoffset{2};
\draw (-3,-1/2) -- (3,1/2);
\draw (-1.5,1) -- (1.5,-1);
\draw plot [smooth, tension=1.3] coordinates {(-2,0.3) (-0,-1) (2,2)};

\draw [fill=gray!30, fill opacity=0.5] (-1.5,1) -- (-1.5,1+\yoffset) -- (1.5,-1+\yoffset) -- (1.5,-1) -- cycle;
\draw [fill=gray!70, fill opacity=0.5] (-3,-1/2) -- (-3,-1/2+\yoffset) -- (3,1/2+\yoffset) -- (3,1/2) -- cycle;
\draw [fill=gray!100, fill opacity=0.5] plot [smooth, tension=1.3] coordinates {(-2,0.3) (-0,-1) (2,2)} -- (2,2+\yoffset) -- plot [smooth, tension=1.3] coordinates {(2,2+\yoffset) (0,-1+\yoffset) (-2,0.3+\yoffset)} -- (-2,0.3);

\draw plot [smooth, tension=1.3] coordinates {(-2,0.3) (-0,-1) (2,2)} -- (2,2+\yoffset) -- plot [smooth, tension=1.3] coordinates {(2,2+\yoffset) (0,-1+\yoffset) (-2,0.3+\yoffset)} -- (-2,0.3);
\draw (-3,-1/2) -- (-3,-1/2+\yoffset) -- (3,1/2+\yoffset) -- (3,1/2) -- cycle;
\draw (-1.5,1) -- (-1.5,1+\yoffset) -- (1.5,-1+\yoffset) -- (1.5,-1) -- cycle;

\draw (0,0) -- (0,0+\yoffset);
\draw (1.27,0.22) -- (1.27,0.22+\yoffset);
\draw (-1.73,-0.29) -- (-1.73,-0.29+\yoffset);
\draw (0.73,-0.48) -- (0.73,-0.48+\yoffset);

\fill (0,0) circle [radius=1.2pt];
\fill (1.27,0.22) circle [radius=1.2pt];
\fill (-1.73,-0.29) circle [radius=1.2pt];
\fill (0.73,-0.48) circle [radius=1.2pt];
\fill (0,0+\yoffset) circle [radius=1.2pt];
\fill (1.27,0.22+\yoffset) circle [radius=1.2pt];
\fill (-1.73,-0.29+\yoffset) circle [radius=1.2pt];
\fill (0.73,-0.48+\yoffset) circle [radius=1.2pt];

\draw [very thick] (-0.59,-1.12) -- (-0.59,-1.12+\yoffset);
\draw [very thick] (-1.14,0.76) -- (-1.14,0.76+\yoffset);
\draw [very thick] (2.5,5/12) -- (2.5,5/12+\yoffset);

\def\arrowshift{0.12};
\def\run{0.3};
\def\slope{1/6};
\draw [<->] (2.5-\run,5/12-\slope*\run-\arrowshift) -- (2.5+\run,5/12+\slope*\run-\arrowshift);
\def\run{0.25};
\def\slope{-2/3};
\draw [<->] (-1.14-\run,0.76-\slope*\run+\arrowshift+\yoffset) -- (-1.14+\run,0.76+\slope*\run+\arrowshift+\yoffset);
\def\run{0.3};
\def\slope{0.15};
\draw [<->] plot [smooth, tension=0.6] coordinates {(-0.59-\run,-1.12+1.5*\slope*\run-\arrowshift) (-0.59,-1.12-\arrowshift) (-0.59+\run,-1.12+\slope*\run-\arrowshift)};

\def\fudge{0.1}
\node at (-2.6,-0.13+\yoffset+\fudge) {$E$};
\node at (-0.3+\fudge,0.6+\yoffset+2*\fudge) {$\phi(E)$};
\node at (1.4-3.6*\fudge,1.35+\yoffset+2*\fudge) {$\phi^2(E)$};
\node [below] at (2.5,5/12-\arrowshift) {$\xi$};
\node [above] at (-1.14,0.76+\yoffset+\arrowshift) {$\xi^\prime$};
\node [below] at (-0.59,-1.12-\arrowshift) {$\xi^{\prime\prime}$};

\def\downoffset{3.25}

\draw (-3,-1/2-\downoffset) -- (3,1/2-\downoffset);
\draw (-1.5,1-\downoffset) -- (1.5,-1-\downoffset);
\draw plot [smooth, tension=1.3] coordinates {(-2,0.3-\downoffset) (-0,-1-\downoffset) (2,2-\downoffset)};

\fill (0,0-\downoffset) circle [radius=1.2pt];
\fill (1.27,0.22-\downoffset) circle [radius=1.2pt];
\fill (-1.73,-0.29-\downoffset) circle [radius=1.2pt];
\fill (0.73,-0.48-\downoffset) circle [radius=1.2pt];

\fill (-0.59,-1.12-\downoffset) circle [radius=2pt];
\fill (-1.14,0.76-\downoffset) circle [radius=2pt];
\fill (2.5,5/12-\downoffset) circle [radius=2pt];

\node [above] at (0,1.7+\yoffset) {$X$};
\node [below] at (0,-1.3-\downoffset) {$\Hilb_{[\xi]}(X)$};
\node [right] at (0.2,-1/2-\downoffset/2) {$\rho$};
\draw [dashed,->] (0.2,-1.3)--(0.2,0.3-\downoffset);

\end{tikzpicture}
\caption{Case 1: Deformations of \(\xi\) determine a map to a surface}
\end{figure}

The second case is when the ray \(\Ras\alpha\) is represented by only finitely many curves. In this setting, the infinitely many contractible divisors \(\phi^n(E)\) must intersect only along a finite number of curves \(\nu_i \subset X\). We show in Section~\ref{localdynamics} that this is impossible.  The crux of the argument is a local dynamical result, given earlier as Theorem~\ref{separateiterates}.

Roughly speaking, we show that there is a uniform bound on the orders of tangency between the divisors \(\phi^m(E)\) and \(\phi^n(E)\) along each curve \(\nu_i\), independent of \(m\) and \(n\).  Then there exists a sequence of blow-ups \(\pi : Y \to X\) centered above the curves \(\nu_i\), such that the strict transforms of the infinitely many divisors \(\phi^n(E)\) all become disjoint on \(Y\).  But this is impossible: each of these divisors is negative on some curve contained in it, contradicting the finite-dimensionality of \(N^1(Y)\).

\begin{figure}[htb]
\label{erigidcurve}
  \centering
\begin{tikzpicture}[scale=0.7]

\def\splitwidth{5}
\def\yoffset{2};

\begin{scope}[shift={(-\splitwidth,0)}]

\def\slope{0.5};
\def\sheetoffsetx{1};
\def\sheetoffsety{\slope*\sheetoffsetx}

\def\overshoot{1}

\draw [fill=gray!130]
(2*\sheetoffsetx,2*\sheetoffsety) -- (2*\sheetoffsetx,2*\sheetoffsety+\yoffset) --
(2*\sheetoffsetx+\overshoot,2*\sheetoffsety+\overshoot*\slope+\yoffset) -- (2*\sheetoffsetx+\overshoot,2*\sheetoffsety+\overshoot*\slope) -- cycle;
\draw [fill=gray!30] (-2.5+2*\sheetoffsetx,0.6+2*\sheetoffsety) -- (-2.5+2*\sheetoffsetx,0.6+2*\sheetoffsety+\yoffset) -- (2.5+2*\sheetoffsetx,-0.6+2*\sheetoffsety+\yoffset) -- (2.5+2*\sheetoffsetx,-0.6+2*\sheetoffsety) -- cycle;
\draw [fill=gray!130]
(\sheetoffsetx,\sheetoffsety) -- (\sheetoffsetx,\sheetoffsety+\yoffset) -- (2*\sheetoffsetx,2*\sheetoffsety+\yoffset) -- (2*\sheetoffsetx,2*\sheetoffsety) -- 
cycle;
\draw [fill=gray!60] (-2.5+\sheetoffsetx,0.6+\sheetoffsety) -- (-2.5+\sheetoffsetx,0.6+\sheetoffsety+\yoffset) -- (2.5+\sheetoffsetx,-0.6+\sheetoffsety+\yoffset) -- (2.5+\sheetoffsetx,-0.6+\sheetoffsety) -- cycle;
\draw [fill=gray!130]
(0*\sheetoffsetx,0*\sheetoffsety) -- (0*\sheetoffsetx,0*\sheetoffsety+\yoffset) -- (1*\sheetoffsetx,1*\sheetoffsety+\yoffset) -- (\sheetoffsetx,\sheetoffsety) -- 
cycle;
\draw [fill=gray!90] (-2.5,0.6) -- (-2.5,0.6+\yoffset) -- (2.5,-0.6+\yoffset) -- (2.5,-0.6) -- cycle;
\draw [fill=gray!130] (0-\overshoot,0-\overshoot*\slope+\yoffset) -- (-\overshoot,-\overshoot*\slope) -- 
(0*\sheetoffsetx,0*\sheetoffsety) -- (0*\sheetoffsetx,0*\sheetoffsety+\yoffset) -- cycle;

\node [above left] at (-2.5,0.6+\yoffset) {$\overline{E}$};
\node [above] at (-2.5+\sheetoffsetx,0.6+\sheetoffsety+\yoffset) {$\overline{\phi(E)}$};
\node [above right] at (-2.5+2*\sheetoffsetx,0.6+2*\sheetoffsety+\yoffset) {$\overline{\phi^2(E)}$};
\node [above] at (-1.2+3*\sheetoffsetx,0.6+3*\sheetoffsety+\yoffset) {$\ldots$};

\end{scope}

\draw [->] (0,1.5) -- (1.5,1.5);
\node [above] at (0.75,1.5) {$\pi$};

\begin{scope}[shift={(\splitwidth,0.5)}]

\def\yoffset{2};

\draw [fill=gray!30] (-1.5,1) -- (-1.5,1+\yoffset) -- (0,\yoffset) -- (0,0) -- cycle;
\draw [fill=gray!60] (-2.5,0.6) -- (-2.5,0.6+\yoffset) -- (0,0+\yoffset) -- (0,0) -- cycle;
\draw [fill=gray!90] (-3,-1/2) -- (-3,-1/2+\yoffset) -- (0,0\yoffset) -- (0,0) -- cycle;
\draw [fill=gray!90] (3,1/2) -- (3,1/2+\yoffset) -- (0,0\yoffset) -- (0,0) -- cycle;
\draw [fill=gray!60] (2.5,-0.6) -- (2.5,-0.6+\yoffset) -- (0,0+\yoffset) -- (0,0) -- cycle;
\draw [fill=gray!30] (1.5,-1) -- (1.5,-1+\yoffset) -- (0,\yoffset) -- (0,0) -- cycle;

\draw [very thick] (0,0) -- (0,0+\yoffset);
\fill (0,0) circle [radius=1.2pt];
\fill (0,0+\yoffset) circle [radius=1.2pt];

\node [above] at (-3,-1/2+\yoffset) {$E$};
\node [above] at (-2.5,0.6+\yoffset) {$\phi(E)$};
\node [above right] at (-1.5,1+\yoffset) {$\phi^2(E)$};

\node [above] at (0.8,1.1+\yoffset) {$\cdots$};
\node [above] at (0,0+\yoffset) {$\nu$};
\end{scope}
\end{tikzpicture}
\caption{Case 2: Separating the divisors \(\phi^n(E)\) by a blow-up}
\end{figure}
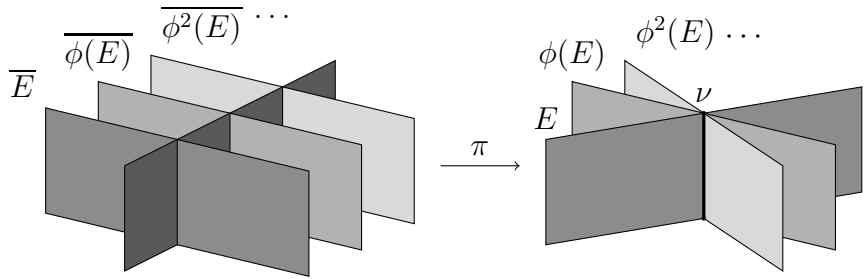

\section{Examples}
\label{examplelist}

We now collect some examples illustrating the conclusions of the theorem, as well as the necessity of the hypotheses.  We begin with a few examples of automorphisms of surfaces, the building blocks for many examples on threefolds.

\begin{example}
\label{abeliansurfaceex}
Let \(E \cong \C/\Lambda\) be an elliptic curve, and let \(A = E
\times E\) be an abelian surface.  There is an action of \(\SL_2(\Z)\) on
\(A\) by automorphisms.  If \(M \in \SL_2(\Z)\) has an eigenvalue greater than \(1\), then the induced automorphism \(\phi : A \to A\) is of positive entropy.
\end{example}

\begin{example}
\label{kummersurfaceex}
Let \(A\) be as in Example~\ref{abeliansurfaceex}, and  \(i : A \to A\) be the  involution \(x \mapsto -x\). The map \(\phi\) above descends to an automorphism \(\psi : A/i \to A/i\).  The quotient \(A/i\) has sixteen nodes, but \(\psi\) lifts to a map \(\bar{\psi} : S \to S\), where \(S\) is a Kummer surface, the minimal resolution of \(A/i\). This is a positive entropy automorphism of a K3 surface.
\end{example}

\begin{example}[\cite{blanc}]
\label{blancex}
Let \(E \subset \P^2\) be a smooth cubic curve and fix a general point \(p\) on \(E\).  Given a general point \(x \in \P^2\), the line \(\ell_{px}\) meets \(E\) at a third, distinct point \(y\).  Define a rational map \(\tau_p : \P^2 \rat \P^2\) which acts on each line \(\ell_{px}\) by the unique involution of \(\P^1\) fixing \(x\) and \(y\).

This map is defined at \(x\) unless the line \(\ell_{px}\) is tangent to \(E\). There are four points \(p_1,p_2,p_3,p_4\) on \(E\) at which such tangency occurs.  One may check that \(\tau_p\) lifts to an involutive automorphism of the blow-up \(X_p = \Bl_{p,p_1,\ldots,p_4} \P^2\), which fixes the strict transform of \(E\) pointwise.

Carrying out the same construction twice more with the same curve \(E\) but using different initial points \(q\) and \(r\) yields automorphisms of the analogous \(5\)-point blow-ups \(X_q\) and \(X_r\), which again fix the strict transform of \(E\) pointwise. The maps \(\tau_p\), \(\tau_q\), and \(\tau_r\) all lift to automorphisms of the common resolution \(X = \Bl_{p,p_i,q,q_i,r,r_i}(\P^2)\), a blow-up of \(\P^2\) at \(15\) points.  Although the three maps are individually involutions, the composition \(\tau_p \circ \tau_q \circ \tau_r\) has positive entropy.
\end{example}

Other constructions give examples of blow-ups of \(\P^2\) at only \(10\) points which admit automorphisms of positive entropy.  These were the first examples of positive entropy automorphisms of rational surfaces, due to Bedford--Kim~\cite{bedfordkim} and McMullen~\cite{mcmullen}.

Examples \ref{kummersurfaceex} and \ref{blancex} typify the two basic constructions of positive entropy automorphisms of rational surfaces~\cite{oguisotruong}:
\begin{enumerate}
\item Start with a positive entropy automorphism \(\psi : Y \to Y\), and a finite order automorphism \(g : Y \to Y\) commuting with \(\psi\).  Then \(\psi\) induces an automorphism of \(Y/\ang{g}\), which lifts to an automorphism \(\phi : X \to X\) of a resolution \(X \to Y/\ang{g}\).  Examples with \(X\) rational can be obtained when \(Y\) is an abelian surface.
\item Start with a carefully chosen birational automorphism \(\psi : Y \rat Y\).  By a sequence of blow-ups, construct a model \(X\) on which \(\psi\) lifts to an automorphism \(\phi : X \to X\).  Many examples with \(X\) rational can be obtained when \(Y = \P^2\).
\end{enumerate}
The result of Theorem~\ref{blowupautos} is that the second of these approaches can never yield primitive automorphisms in dimension \(3\).  In contrast, blow-up constructions do yield many interesting pseudoautomorphisms in higher dimensions, as in e.g.\ \cite{perronizhang}, \cite{bedforddillerkim},\dots.

The next example shows that the two-dimensional analog of Corollary~\ref{invariantstuff} is not true.
\begin{example}[{\cite[Theorem 4.2]{bedfordkim3}}]
\label{bedfordkimnocurves}
There exists a rational surface \(S\) and a positive entropy automorphism \(\phi : S \to S\) with no \(\phi\)-invariant curves.
\end{example}

We now give some three-dimensional examples illustrating the various cases of Theorem~\ref{main} and the more detailed Theorem~\ref{precisemain}.

\begin{example}[Case 1(a)]
The construction in Example~\ref{abeliansurfaceex} generalizes to dimension three. Let \(E\) be an elliptic curve and \(M \in \SL_3(\Z)\) a linear map with an eigenvalue greater than \(1\); then \(M\) induces a positive entropy automorphism of \(A = E \times E \times E\), which can be primitive.
\end{example}

\begin{example}[Case 1(b), \cite{oguisotruong}]
\label{otcalabiyauex}
Let \(\omega = (-1+\sqrt{3}i)/2\) and consider the elliptic curve \(E = \C/(\Z \oplus \Z\omega)\).  Then \(\Z/3\Z\) acts on \(E \times E \times E\) via the map \(\sigma(x,y,z) = (\omega x , \omega y, \omega z)\). The quotient \((E \times E \times E)/\sigma\) has canonical singularities of type \(\nicefrac{1}{3}(1,1,1)\).  Let \(X \to (E \times E \times E)/\sigma\) be the crepant resolution given by blowing up each singular point.  Then \(X\) is a smooth Calabi-Yau threefold which admits a primitive automorphism of positive entropy, induced by an element of \(\SL_3(\Z)\).
\end{example}

\begin{example}[Case 2(a)]
Let \(S\) be a projective K3 surface admitting a positive entropy automorphism \(\psi : S \to S\).  Let \(C\) be a curve of genus at least \(2\), and take \(X = S \times C\) with \(\phi = \psi \times \id\).
Then \(\kappa(X) = 1\), the canonical class \(K_X = p_2^\ast K_C\) is nef, and the canonical model of \(X\) is the projection \(\pi : X \to C\), given by the linear system \(\abs{3K_X}\).  The canonical fibration realizes \(\phi\) as an imprimitive map.
\end{example}

\begin{example}[Case 2(b)]
\label{p1bundle}
Let \(\psi : S \to S\) be a surface automorphism of positive entropy and \(\cE\) be a rank-\(2\) vector bundle on \(S\) for which there exists an isomorphism \(\psi^\ast(\cE) \to \cE\).  Then \(\psi\) induces a positive entropy automorphism of \(X = \P_S (\cE)\). The total space \(X\) is a \(\P^1\)-bundle over \(S\).  For example, take \(\psi : S \to S\) to be an automorphism of a rational surface, and set \(\cE = TS\) or \(\cE = \cO \oplus \cO(K_S)^{\otimes n}\).
\end{example}

\begin{example}[Case 2(c), cf.\ also Theorem~\ref{blowupautos}]
\label{productex}
Let \(Y = \P^2 \times \P^1\).  By blowing up ten curves \(p_i \times \P^1\), we obtain \(X = S \times \P^1\), where \(S\) is a rational surface.  If the points \(p_i\) are chosen carefully, then \(S\) admits a positive entropy automorphism \(\psi : S \to S\), and \(X\) has an automorphism \(\phi = \psi \times \id : X \to X\).  This is consistent with Theorem~\ref{blowupautos}, which shows that if a blow-up of a smooth threefold with no positive entropy automorphisms admits a positive entropy automorphism, the automorphism must be imprimitive.  In this example, we have \(X^\prime = X\) in the formulation of case 2(c).
\end{example}

\begin{example}[Case 3(a)]
Let \(\psi : Y \to Y\) be a positive entropy automorphism of a smooth threefold, and let \(V\) be \(\psi\)-periodic closed subscheme of \(Y\). Then  \(\psi\) lifts to an automorphism of \(X = \Bl_V Y\).  For example, \(V\) might be a \(\psi\)-invariant point or smooth curve.
\end{example}

\begin{example}[Case 3(b), \cite{oguisotruong}]
\label{otrational}
Let \(\omega\) and \(E\) be as in Example~\ref{otcalabiyauex}, and consider the order \(6\) automorphism of \(E \times E \times E\) given by \(\tau(x,y,z) = (-\omega x, -\omega y, -\omega z)\).  Let \(X\) be a resolution of the quotient \((E \times E \times E)/\tau\).  It is checked in~\cite{oguisotruong} that \(X\) a smooth rational threefold and that the action of \(\SL_3(\Z)\) induces primitive automorphisms of positive entropy.   However, these automorphisms are not dynamically minimal in the sense of Definition~\ref{minimaldef}, because some of the exceptional divisors of the resolution \(X \to (E \times E \times E)/\tau\) are \(\phi\)-invariant and can be contracted to terminal singularities.  We explore the geometry of this example in more detail in Section~\ref{singularcase}.
\end{example}

\begin{remark}
Since the variety \(X\) of Example~\ref{otrational} admits an imprimitive automorphism of positive entropy, it follows from Theorem~\ref{blowupautos} that \(X\) can not be obtained by a sequence of smooth blow-ups of \(\P^3\), answering \cite[Question 5.11]{oguisoicm}.  Bisi, Cascini, and Tasin have recently given a much simpler proof of this fact for a broad class of quotients of abelian threefolds.
\end{remark}

\section{Preliminaries}
\label{thehardpart}

To begin, we collect some conventions. 
Suppose that \(X\) and \(Y\) are projective varieties.  A map \(\phi : X \to Y\) written with a solid arrow indicates a morphism, while a map \(\phi : X \rat Y\) denotes a rational map.  By an \emph{automorphism}  \(\phi : X \to X\), we mean a biregular automorphism.  A birational map \(\phi : X \rat X\) is called a \emph{birational automorphism}.  A \emph{pseudoautomorphism} \(\phi : X \rat X\) is a birational automorphism that is an isomorphism in codimension \(1\), i.e.\ such that neither \(\phi\) nor \(\phi^{-1}\) contracts any divisors.

Suppose that \(X\) is a smooth, projective variety. If \(V \subset X\) is a closed subscheme of \(X\), we write \(\Bl_V(X)\) for the blow-up of \(X\) along \(V\).  We will say that \(Y\) is a \emph{smooth blow-up} of \(X\) if there is a sequence of maps \(\pi_i  : X_{i+1} \to X_i\), with \(X_n = Y\) and \(X_0 = X\), such that each map \(\pi_i\) is the blow-up of \(X_i\) along a smooth subvariety. This is a stronger than the assumption that \(Y\) is smooth and can be obtained as the blow-up of some closed subscheme \(V \subset X\)(cf. \cite[Ex.\ 22]{kollarexercises}).

If \(\phi : X \to X\) is an automorphism of a variety, and \(\pi : Y \to X\) is a morphism from some other variety, we say that \(\phi\) \emph{lifts} to an automorphism of \(Y\) if there exists an automorphism \(\psi : Y \to Y\) such that \(\pi \circ \psi = \phi \circ \pi\).  Similarly, if \(\pi : X \to V\) is a morphism from \(X\) to another variety, we say that \(\phi\) \emph{descends} to an automorphism of \(V\) if there exists an automorphism \(\psi : V \to V\) with \(\pi \circ \psi = \phi \circ \pi\).

A normal variety \(X\) is said to have terminal singularities if:
\begin{enumerate}
\item \(mK_X\) is Cartier for some integer \(m\);
\item on a smooth resolution \(\pi : Y \to X\), we can write \(mK_Y = f^\ast(mK_X) + \sum a_i E_i\), and the coefficients \(a_i\) are all positive. Equivalently, given a regular \(n\)-form \(\omega\) on \(X\) (or more generally a section of \(mK_X\) for any \(m > 0\)), the pull-back of \(\omega\) to \(Y\) vanishes along all the exceptional divisors. 
\end{enumerate}

In dimension two, \(X\) has terminal singularities if and only if it is smooth.  In higher dimensions, this class of singularities arises naturally in the course of running the MMP. 

Write \(N^1(X)\) for the \(\R\)-vector space of divisors on \(X\) modulo numerical equivalence, and \(N^1(X)_\Z\) for the lattice in \(N^1(X)\) spanned by divisors with integral coefficients.  The Picard rank \(\rho(X)\) is the dimension of \(N^1(X)\), which is finite.  If \(D\) is a divisor or line bundle on \(X\), we write \([D]\) for its numerical class.   Dually, \(N_1(X)\) is the \(\R\)-vector space of curves on \(X\) modulo numerical equivalence, \([C]\) is the numerical class of a curve \(C\), and \(\NEb(X) \subset N_1(X)\) is the Mori cone, the closure of the span of effective curve classes.

The proof of Theorem~\ref{main} will require some results from the threefold MMP, which we collect below as a single statement.

\begin{theorem}[The cone theorem for smooth threefolds, etc.]
\label{conetheorem}
Suppose that \(X\) is a smooth, projective threefold.  There is a countable set of rational curves \(C_i \subset X\) with \(-4 \leq K_X \cdot C_i < 0\) such that
\[
\NEb(X) = \NEb_{K_X \geq 0}(X) + \sum_i \Ra [C_i].
\]
Let \(R\) be a \(K_X\)-negative extremal ray; if \(K_X\) is not nef, then there exists at least one such ray.  There exists a contraction map \(c_R : X \to Z\) to a projective variety \(Z\)  such that a curve \(C \subset X\) is contracted to a point by \(c_R\) if and only if \([C]\) lies on the ray \(R\). Moreover,  \((c_R)_\ast(\cO_X) = \cO_Z\) and \(\rho(Z) = \rho(X) - 1\).

The contraction \(c_R\) is of one of the following types. 
\begin{enumerate}
\item (Mori fiber space).  We have \(\dim Z < \dim X\), and the general  fiber of \(c_R\) is a Fano variety.  There are three subcases:
\begin{enumerate}
\item \(Z\) is a surface, and the fibers of \(c_R\) are plane conics;
\item \(Z\) is a curve, and general fibers of \(c_R\) are del Pezzo surfaces;
\item \(Z\) is a point, and \(X\) is a Fano variety of Picard rank \(1\).
\end{enumerate}
\item (Divisorial contraction). The map \(c_R : X \to Z\) is birational, and the exceptional locus of \(c_R\) consists of a single irreducible divisor \(E\).  One of the following sub-cases occurs:
\begin{enumerate}
\item[(E1)] \(Z\) is smooth, and \(c_R : X \to Z\) is the blow-up of a smooth curve.
\item[(E2)] \(Z\) is smooth, and \(c_R : X \to Z\) is the blow-up of a smooth point.
\item[(E3)] \(E \cong \P^1 \times \P^1\) with normal bundle of bidegree \((-1,-1)\), and the two rulings on \(E\) are numerically equivalent in \(X\).  The map \(\pi : X \to Z\) contracts \(E\) to a singular point locally analytically isomorphic to \(x^2+y^2+z^2+t^2 = 0\).
\item[(E4)] \(E\) is isomorphic to a singular quadric cone, with normal bundle \(\cO_E(E) = \cO_E \otimes \cO_{\P^3}(-1)\). The image \(Z\) has a singularity locally analytically isomorphic to \(x^2+y^2+z^2+w^3 = 0\).
\item[(E5)] \(E\) is isomorphic to \(\P^2\) with normal bundle \(\cO_{\P^2}(-2)\), and \(c_R\) contracts \(E\) to a singular point. The singularity is locally analytically isomorphic to the vertex of the cone over the Veronese surface in \(\P^5\), the quotient \(\A^3_\C/(\pm 1)\).
\end{enumerate}
\end{enumerate}
Furthermore, the number of rays \(R\) determining contractions of type (1) is finite.
\end{theorem}
\begin{proof}
The first parts are the cone and contraction theorems, which can be found e.g.\ as~\cite[Theorem 3.7]{kollarmori}.  The fact that \(\rho(Z) = \rho(X)-1\) is \cite[Corollary 3.17]{kollarmori}.  The classification of contractions on a smooth threefold is a fundamental result of Mori~\cite{morinotnef}; the breakdown of case (2) into subcases appears as~\cite[Theorem 3.3]{morinotnef}.  Note that on a smooth threefold, there are no flipping contractions.

The final claim on the number of rays giving Mori fiber spaces is an observation of Wi\'sniewski~\cite[Theorem 2.2]{wisniewski} (see also ~\cite{kollarrationalcurves}, Exercise III.1.9). We include the proof for convenience.

Let \(V \subset N^1(X)\) denote the affine cubic hypersurface defined by \(D^3 = 0\).  If \(c_R : X \to Z\) is the contraction of a \(K_X\)-negative extremal ray, then \(\rho(Z) = \rho(X)-1\), and in particular \(c_R^\ast(N^1(Z)) \subset N^1(X)\) has codimension \(1\).  If \(R\) determines a Mori fiber space, then \(\dim Z < \dim X\) and so \((f^\ast D)^3 = 0\) for any class \(D \in N^1(Z)\).  In particular \(f^\ast(N^1(Z)) \subset V\) is a hyperplane contained in \(V\).  However, \(V\) is a degree \(3\) affine subvariety, and so contains at most \(3\) hyperplanes.  The number of extremal rays determining Mori fiber space structures is thus at most \(3\).
\end{proof}

\begin{remark}
The reader interested primarily in automorphisms of smooth blow-ups, as in Theorem~\ref{blowupautos}, need not worry about the MMP. When we consider a contraction \(\pi : X \to Z\) of the MMP, it can be assumed to be the final of the sequence of blow-ups used in constructing \(X\), so that \(\pi : X \to Z\) is either the blow-up of a point or a smooth curve.  These correspond to contractions of Type (E1) or (E2) in Theorem~\ref{conetheorem}.  Lemmas~\ref{nefcanonical} and~\ref{morifiberspace} below are not needed in this case, but the rest of the argument is essentially the same.
\end{remark}

\begin{lemma}
\label{autosdescend}
Suppose that \(\phi : X \to X\) is a positive entropy automorphism, and that \(c_R : X \to Z\) is the contraction of an extremal ray on \(X\).  If \(\phi^\ast(R) = R\), then \(\phi\) descends to a automorphism \(\psi : Z \to Z\).  If \(c_R\) is a divisorial contraction, then \(\psi\) has positive entropy as well.
\end{lemma}
\begin{proof}
The composition \(c_R \circ \phi\) contracts every curve with numerical class on the ray \(\phi^{\ast}(R) = R\), and so has the same fibers as \(c_R\) itself.  Since \((c_R)_\ast(\cO_X) = \cO_Z\), by the rigidity lemma, \(c_R \circ \phi\) factors through \(c_R\), inducing a map \(\psi : Z \to Z\)~\cite[Lemma 1.15(b)]{debarre}.
Applying the same argument with \(c_R\) and \(c_R \circ \phi\) exchanged shows that the induced map on \(Z\) is an automorphism.
\[
\xymatrix{
X \ar[d]_{c_R} \ar[r]^\phi & X \ar[d]_{c_R} \\
Z \ar[r]^\psi & Z
}
\]
If \(c_R\) is divisorial, we have a decomposition \(N^1(X) = c_R^\ast N^1(Z) \oplus \R[E]\), where \(E\) is the exceptional divisor.  The divisor \(E\) is \(\phi\)-invariant, so the block form of \(\phi^\ast\) with respect to this decomposition is \(\phi^\ast = \left( \begin{smallmatrix} \psi^\ast & 0 \\ - & 1 \end{smallmatrix} \right)\).  In particular, the eigenvalues of \(\psi^\ast\) coincide with those of \(\phi^\ast\), but without one eigenvalue \(1\).  As \(\phi^\ast\) has an eigenvalue bigger than \(1\), so too must \(\psi^\ast\).
\end{proof}

The geometry of ruled surfaces contained in \(X\) plays an essential role in the argument, and we next recall some basic facts about cones of divisors on ruled surfaces.  By a ruled surface we mean the projectivization of a rank-\(2\) bundle over a smooth curve, sometimes called a geometrically ruled surface.

\begin{proposition}
\label{ruledsurfacefacts}
Suppose that \(C\) is a smooth curve, \(\cE\) is a rank-\(2\) vector bundle over \(C\) and \(S = \P_C(\cE)\), with projection \(g : S \to C\) and general fiber \(f\).  Then \(N^1(S)\) is generated by two classes: the class \([f]\) of a fiber, and the class \(\xi = [\cO_S(1)]\).  The cone of curves \(\NEb(S) \subset N_1(S)\) is spanned by two boundary classes: the class \([f]\) of a fiber, and a second class \(\alpha\).  The ray \(\Ras \alpha\) satisfies one of the following:
\begin{enumerate}
\item[(R1).] \(\alpha^2 < 0\) and \(\Ras \alpha\) is represented by a unique irreducible curve.
\item[(R2).] \(\alpha^2 = 0\) and either:
\begin{enumerate}
\item[(R2a).] there is only a finite set of curves with numerical class in \(\Ras \alpha\);
\item[(R2b).] there is a map \(h : S \to \P^1\) such that the fibers of \(h\) are all in the class \(\alpha\).  Every irreducible curve with class in \(\Ras \alpha\) is a rational multiple of a fiber of \(h\). 
\end{enumerate}
\end{enumerate}
Dually, the nef cone is spanned by \([f]\) and a second ray \(\beta\) satisfying \(\alpha \cdot \beta = 0\).  Both \(\NEb(S)\) are \(\Nef(S)\) are rational polyhedral cones. In Case (R2), the rays \(\alpha\) and \(\beta\) coincide. 
\end{proposition}
\begin{proof}
Because \(\alpha\) spans an extremal ray on \(\NEb(S)\), it must be that \(\alpha^2 \leq 0\).  Moreover, if \(\alpha^2 < 0\), then the ray \(\Ras \alpha\) is spanned by the class of an irreducible curve and there is only a single curve \(B\) with \([B] \in \Ras \alpha\)~\cite[Lemma 6.2(d,e)]{debarre}.

It remains to consider the case in which \(\alpha^2 = 0\).  Suppose that there exist three irreducible curves \(B_1\), \(B_2\), and \(B_3\) whose classes lie on the ray \(\Ras \alpha\). Since \(\alpha^2 = 0\), these curves are necessarily pairwise disjoint, and \cite[Theorem 2.1]{totaroalbanese} implies that there exists a map \(h : S \to \Gamma\) with each of the curves \(B_i\) a multiple of a fiber of \(h\).  The fibers of \(g : S \to C\) are rational curves which are not contracted by \(h\), so it must be that \(\Gamma \cong \P^1\). Every curve with class on \(\Ras \alpha\) must be a fiber of \(h\).

If \(\alpha^2 < 0\), then \(\Ras \alpha\) is represented by an irreducible curve, and the ray \(\Ras \alpha\) is certainly rational.  If \(\alpha^2 = 0\), write \(\alpha \sim a f + \xi\), and then \( \alpha^2 = (af+\xi)^2 = 2a + \xi^2 = 2a + \deg \cE\).  This gives \(a = -\deg \cE/2\), and the ray is again rational.  The dual statements for the nef cone are immediate.
\end{proof}

A ruled surface of type (R1) corresponds to the case that \(\cE\) is unstable, while (R2) arises when \(\cE\) is semistable. In case (R2), we will say that an irreducible curve \(C\) is an \(S\)-covering curve if \([C]\) lies on the ray \(\Ras \alpha\), and \(C\) moves in a family covering \(S\).  We say that \(C\) is an \(S\)-rigid curve if \([C]\) is on the ray \(\Ras \alpha\), but \(C\) does not move in a family.  In Case (R2a), any \(C\) with class on \(\Ras \alpha\) is \(S\)-rigid.  In Case (R2b), both \(S\)-covering and \(S\)-rigid curves can occur: a general fiber of \(h\) is \(S\)-covering, while the support of a multiple fiber of \(h\) is an \(S\)-rigid curve.  In either case, the number of \(S\)-rigid curves is finite.

\begin{figure}[htb]
  \centering
\begin{tikzpicture}[dot/.style={circle,fill=black,minimum size=3pt,inner sep=0pt,outer sep=-1pt}]

\fill[gray!30] (0,2) -- (2,2) -- (2,-1) -- (0,0) -- cycle;
\fill[gray!50] (0,2) -- (2,2) -- (2,1) -- (0,0) -- cycle;

\draw [<->,thick] (-2.1,0) -- (2.1,0);
\draw [<->,thick] (0,-2.1) -- (0,2.1);
\foreach \xx[evaluate={\x=\xx/2}] in {-4,-3,...,4} {
  \draw (\x,2pt) -- (\x,-2pt);
  \draw (2pt,\x) -- (-2pt,\x);
}

\draw [thick] (0,0)--(0,2);
\draw [thick] (0,0)--(2,1);
\draw [thick] (0,0)--(2,-1);

\node [above right] (C1) at (2,0) {$\xi$};
\node [above right] (C2) at (0,2) {$[f]$};
\node (NF) at (3/4,1) {$\scriptstyle \Nef(S)$};
\node (EF) at (3/2,-1/3) {$\scriptstyle \NEb(S)$};
\node [right] (C2) at (2,1) {$\beta$};
\node [right] (C2) at (2,-1) {$\alpha$};

\end{tikzpicture}
\caption{Cones in $N^1(S)$}
\end{figure}
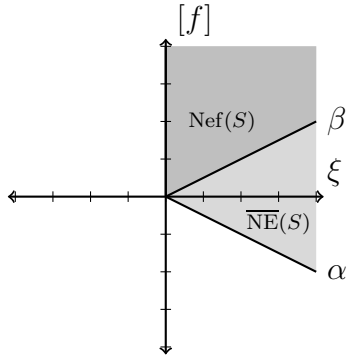

\begin{example}
\label{lotsofruledsurfaces}
We recall some examples of ruled surfaces to illustrate the various possibilities.
\begin{enumerate}
\item (R1) Let \(S = \P_{\P^1}(\cO_{\P^1} \oplus \cO_{\P^1}(n))\) with \(n
\geq 1\) be a Hirzebruch surface.  There is a unique curve of
negative self-intersection representing the class \(\alpha\).
\item (R2a) Let \(C\) be a curve of genus at least \(2\), and let \(\cE\) be a general semistable rank \(2\) bundle on \(C\).  The ray \(\alpha\) on \(\NEb(\P_C(\cE))\) is not represented by any curve~\cite[Example 1.5.1]{lazarsfeld}.
\item (R2a) Let \(E\) be an elliptic curve and \(\cE\) be the non-split extension of \(\cO_E\) by \(\cO_E\).  The ruled surface \(S = \P_E(\cE) \to E\) has a section determined by \(\cE \to \cO_E\), which is the unique curve representing the ray \(\Ras \alpha\). 
\item (R2a) Let \(E\) be an elliptic curve and let \(L\) be a degree \(0\) nontorsion line bundle on \(E\).  Consider the rank-\(2\) bundle \(\cE = \cO \oplus L\).  The ruled surface \(S = \P_E(\cE)\) has two sections, determined by the quotients \(\cE \to \cO\) and \(\cE \to L\), each representing the class \(\alpha\).  There are no other curves with class on \(\Ras \alpha\).
\item (R2b) Let \(S = C \times \P^1 \to C\).  The class \(\alpha\) is represented by all sections \(C \times x\), which are \(S\)-covering curves.
\item (R2b) Let \(E\) be an elliptic curve and let \(M\) be a degree \(0\) \(n\)-torsion line bundle with \(n \geq 2\).  As in (4), there are two sections \(B_1\), \(B_2\) of \(S\), with normal bundles \(M\) and \(M^\ast\).  These sections are \(S\)-rigid curves. There is a map \(h : S \to \P^1\) whose general fibers are \(n\)-fold multisections of \(S \to E\).  The general fibers are \(S\)-covering curves.  The curves \(B_1\) and \(B_2\) appear as the supports of the multiple fibers of \(h\).
\end{enumerate}
\end{example}

\section{Geometric consequences of positive entropy}
\label{geometric}

We are now in position to begin the proof of Theorem~\ref{precisemain}.  Suppose that \(X\) is a smooth threefold, and \(\phi : X \to X\) is an automorphism of positive entropy.

\begin{lemma}[\cite{dqzhangmmp}]
\label{nefcanonical}
If \(K_X\) is nef, then \(\phi : X \to X\) satisfies Theorem~\ref{precisemain}.
\end{lemma}
\begin{proof}
Since \(K_X\) is nef, the abundance theorem in dimension \(3\)~\cite{kawamataabundance} implies that \(K_X\) is semiample and \(\phi : X \to X\) preserves the canonical fibration \[X \to X_\mathrm{can} = \Proj \bigoplus_{m \geq 0} H^0(X,\cO_X(mK_X)).\]  If \(1 \leq \kappa(X) \leq 2\), then case 2(a) of Theorem~\ref{precisemain} is satisfied.  If \(\kappa(X) = 3\), then \(X\) is of general type and has finite birational automorphism group, and in particular can not admit any positive entropy automorphism.
This shows that \(\phi\) must be imprimitive unless \(\kappa(X) = 0\). 

If \(K_X\) is nef and \(\kappa(X) = 0\), then \(K_X\) must be numerically trivial, and \(\phi\) satisfies Case (1) of Theorem~\ref{precisemain}.  The breakdown into subcases is an observation of Zhang~\cite{dqzhangmmp}.  Consider the Albanese map \(\Alb_X : X \to \Alb(X)\). Since \(\kappa(X) = 0\), \(\Alb_X\) is surjective with connected fibers by a result of Kawamata~\cite{kawamataalb}.  The automorphism \(\phi\) is imprimitive with respect to \(\Alb_X\) unless \(h^{0,1}(X) = 0\) or \(h^{0,1}(X) = 3\).  In the former case, \(X\) is a weak Calabi-Yau variety, which is case 1(b) of Theorem~\ref{precisemain}.  In the latter case, \(\Alb_X\) must be birational, and since \(K_X\) is nef, it must be that \(X = \Alb(X)\) and case 1(a) is satisfied.
\end{proof}

From now on, we will assume that \(K_X\) is not nef.  By the cone theorem, there must exist a \(K_X\)-negative extremal ray \(R\) on \(\NEb(X)\), and a contraction morphism \(\pi : X \to Y\).

\begin{lemma}
\label{morifiberspace}
Suppose that \(X\) admits the structure of a Mori fiber space \(\pi : X \to V\).  Then after replacing by an iterate, any automorphism \(\phi : X \to X\) of infinite order is not primitive, and \(\phi\) satisfies Case 2(b) of Theorem~\ref{precisemain}.
\end{lemma}
\begin{proof}
By the final claim of Theorem~\ref{conetheorem}, the number of extremal rays which determine Mori fiber contractions is finite, and after replacing \(\phi\) by an appropriate iterate we may assume that the ray \(R\) determining \(\pi\) is fixed by \(\phi^\ast\).  By Lemma~\ref{autosdescend}, \(\phi\) descends to an automorphism of \(V\).

If the image of the \(\pi : X \to V\) is a single point and \(\rho(X/V) = 1\), it must be that \(X\) is a Fano variety of Picard rank \(1\).  But the condition of Picard rank \(1\) is incompatible with the existence of a positive entropy automorphism.  If the image of \(\pi : X \to V\) is a curve, then \(\rho(X) = \rho(V) + 1 = 2\).  This too is incompatible with the existence of a positive entropy automorphism, because \(K_X\) is a nonzero \(1\)-eigenvector of \(\phi^\ast : N^1(X) \to N^1(X)\) (as \(X\) is uniruled) and \(\phi^\ast\) has determinant \(\pm 1\).

The final case is that \(\pi : X \to V\) is a map to a surface, and the general fiber is \(\P^1\).  By Mori's classification of threefold contractions, \(p\) must be a conic bundle.  This is Case 2(b) of Theorem~\ref{precisemain}.
\end{proof}

\begin{lemma}
\label{periodicexceptional}
Suppose that \(X\) admits a divisorial contraction \(\pi : X \to Y\) with exceptional divisor \(E\), corresponding to the contraction of an extremal ray \(R \subset \NEb(X)\).  If \(E\) is \(\phi\)-periodic, then some iterate \(\phi^n\) descends to an automorphism of \(Y\).   The map \(\phi\) satisfies either Case 3(a) or Case 3(b) of Theorem~\ref{precisemain}.
\end{lemma}
\begin{proof}
Suppose that \(E\) is \(\phi\)-periodic. Replacing \(\phi\) by \(\phi^n\), we may assume that \(\phi(E) = E\).  The map \(\phi\vert_E : E \to E\) is an automorphism of the exceptional divisor.  If the image of \(N_1(E) \to N_1(X)\) is \(1\)-dimensional (as in cases (E2), (E3), (E4), and (E5) of Theorem~\ref{conetheorem}), then \(\phi^n\) fixes the ray \(R\) and \(\phi^n\) descends to an automorphism of \(Y\) by Lemma~\ref{autosdescend}.

If the image of \(N_1(E) \to N_1(X)\) is \(2\)-dimensional, then the restriction \(\phi\vert_E^\ast : N_1(E) \to N_1(E)\) either fixes both boundary rays on \(\NEb(E)\), or exchanges the two rays.  Replacing \(\phi\) by \(\phi^2\) if needed, we may assume that \(\phi\vert_E\) acts by the identity on \(N_1(E)\) and so fixes \(R\). By Lemma~\ref{autosdescend}, \(\phi\) descends to a positive entropy automorphism of \(Y\).  

If \(\pi\) is of type (E1) or (E2) in the classification of Theorem~\ref{conetheorem}, then \(Y\) is smooth and \(\phi\) satisfies Case 3(a) of Theorem~\ref{precisemain}.  If \(\pi\) is of type (E3), (E4), or (E5), then \(Y\) satisfies Case 3(b).
\end{proof}

\begin{remark}
When \(E\) is \(\phi\)-periodic this is essentially a step in the \(\phi\)-equivariant MMP.  However, when \(\phi\) has infinite order, there might not exist a \(\phi^\ast\)-invariant \(K_X\)-negative extremal ray.
\end{remark}

For simplicity, let us give a name to the following condition on a positive entropy automorphism \(\phi\):
\begin{itemize}
\item[(A)] There exists a divisorial contraction \(\pi : X \to Y\), with exceptional divisor \(E\), such that \(E\) is not \(\phi\)-periodic.
\end{itemize}
In this setting, we write \(E_n\) for the divisor \(\phi^n(E)\), and \(f_n = \phi^n(f) \subset E_n\) for a fiber of the map \(\pi \circ \phi^{-n} : X \to Y\) contracting \(E_n\).

\begin{remark}
If \(X\) is a smooth threefold of non-negative Kodaira dimension, every \(K_X\)-negative extremal contraction is divisorial.  The only divisors that can be contracted are those in the stable base locus of \(\abs{K_X}\), and there are only finitely many such divisors.  Consequently, Condition (A) can never hold if \(\kappa(X) \geq 0\), so this condition implies that \(X\) is uniruled. If \(X\) is uniruled but not rationally connected, then any automorphism \(\phi : X \to X\) descends to a birational automorphism of the target of the mrc fibration on \(X\) and so is imprimitive~\cite{dqzhangmmp}.   The results that follow are mostly of interest when \(X\) is rationally connected, although they hold in general.
\end{remark}

Lemmas~\ref{nefcanonical}, \ref{morifiberspace}, and \ref{periodicexceptional} show that if \(\phi\) does not satisfy Condition (A), then \(\phi\) satisfies the claims of Theorem~\ref{precisemain}. It remains to prove the theorem when \(\phi\) does satisfy Condition (A). In fact, we will show that in this case \(\phi\) must be imprimitive.  We next observe that the divisors \(\phi^n(E)\) must have nonempty intersection.
\begin{lemma}
\label{intersections}
Suppose that \(\phi : X \to X\) is an automorphism satisfying Condition (A). Then there are infinitely many values of \(n\) for which \(E_n \cap E\) is nonempty. 
\end{lemma}
\begin{proof}
Suppose that \(E_n \cap E\) is nonempty for only finitely many values of \(n\).  Then there is some \(N\) for which \(E_n \cap E\) is empty for any \(n\) with \(\abs{n} \geq N\).  Replacing \(\phi\) by the iterate \(\phi^N\), we may assume that \(E_n \cap E\) is empty for all \(n\).  Then \(E_m \cap E_n = \phi^m(E \cap E_{n-m})\) is also empty for any distinct \(m\) and \(n\).  Since \(E \cdot f < 0\), we have \(E_m \cdot f_m < 0\). However, \(E_m \cdot f_n = 0\) for \(m \neq n\) because \(E_m\) and \(E_n\) are disjoint. This implies the classes of the infinitely many \(E_m\) are linearly independent in \(N^1(X)\), contradicting the finite-dimensionality of \(N^1(X)\).
\end{proof}
\begin{lemma}
\label{notpoint}
Suppose that \(\phi : X \to X\) is a positive entropy automorphism satisfying Condition (A).  Then the image of \(E\) under \(\pi : X \to Y\) is not a point.
\end{lemma}
\begin{proof}
By Lemma~\ref{intersections}, we may replace \(\phi\) by an iterate and assume that \(E_1\) has nonempty intersection with \(E\).  Suppose that the map \(\pi : X \to Y\) contracts \(E_1 \cap E \subset E\) to a point. Let \(C\) be a curve contained in this intersection.  Since \(\pi\) is a divisorial contraction of a ray \(R \subset \NEb(X)\) and \(C\) is contracted to a point by \(\pi\), the class \([C]\) lies on \(R\).  On the other hand, \(C\) is contained in \(E_1\) and so is contracted by \(\pi \circ \phi^{-1}\), so \([C]\) lies on \(\phi^\ast(R)\).  But \(R\) and \(\phi^\ast(R)\) are distinct rays, a contradiction. 
\end{proof}
Contractions of types (E2), (E3), (E4), and (E5) all contract a divisor to a point, so if \(\phi : X \to X\) satisfies Condition (A), the contraction \(\pi : X \to Y\) must be of type (E1), so that \(Y\) is smooth and \(\pi\) is the blow-up of a smooth curve in \(Y\).  Lemma~\ref{notpoint} provides another proof of the familiar fact that if \(M\) has no positive entropy automorphisms, then no variety obtained by blowing up a set of points in \(M\) can have a positive entropy automorphism.

Next we collect some observations about properties of the leading eigenvector of \(\phi^\ast : N^1(X) \to N^1(X)\).

\begin{lemma}[\cite{truong}]
\label{numdim}
Let \(\lambda = \lambda_1(\phi)\) be the spectral radius of \(\phi^\ast : N^1(X) \to N^1(X)\). After replacing \(\phi\) with \(\phi^{-1}\) if necessary, there exists an \(\R\)-divisor class \(D\) such that 
\begin{enumerate}
\item \(D\) is nef and \(\phi^\ast D = \lambda D\),
\item \(D\) is has numerical dimension \(1\) (i.e. \(D^2 = 0\)),
\item \(D\) is not a multiple of any rational class. 
\end{enumerate}
Moreover, all of these properties hold even after replacing \(\phi\) by any positive iterate \(\phi^n\).
\end{lemma}
\begin{proof}
The map \(\phi^\ast : N^1(X)_\Z \to N^1(X)_\Z\) and its inverse \((\phi^{-1})^\ast : N^1(X)_\Z \to N^1(X)_\Z\) are both defined by integer matrices, and so both have determinant \(\pm 1\).  If \(\phi^\ast\) has an eigenvalue of norm greater than \(1\), then it also has one of norm less than \(1\), which implies that \(\phi^{-1}\) is of positive entropy as well; we may thus later replace \(\phi\) with \(\phi^{-1}\) and retain the assumption of positive entropy. 

The map \(\phi^\ast\) preserves the strongly convex cone \(\Nef(X) \subset N^1(X)\), and by a standard form of the Perron-Frobenius theorem \(\lambda\) is in fact a real eigenvalue of \(\phi^\ast\), and there exists a nef class \(D\) with \(\phi^\ast D = \lambda D\)~\cite{birkhoff}. Let \(\lambda^\prime = \lambda_1(\phi^{-1})\) be the spectral radius of \((\phi^{-1})^\ast : N^1(X) \to N^1(X)\).  By the same argument, there is a nef class \(D^\prime\) with \((\phi^{-1})^\ast(D^\prime) = \lambda^\prime D^\prime\).

Suppose that \(D^2\) and \((D^{\prime})^2\) are both nonzero.  Then \(D^2\) is an eigenvector of \(\phi^\ast : N^2(X) \to N^2(X)\) with eigenvalue \(\lambda_1(\phi)^2\), and \((D^\prime)^2\) is an eigenvector of \((\phi^{-1})^\ast : N^2(X) \to N^2(X)\) with eigenvalue \(\lambda_1(\phi^{-1})^2\). This yields the two inequalities \(\lambda_1(\phi)^2 \leq \lambda_2(\phi)\) and \(\lambda_1(\phi^{-1})^2 \leq \lambda_2(\phi^{-1})\).  The maps \(\phi^\ast : N^1(X) \to N^1(X)\) and \((\phi^{-1})^\ast : N^2(X) \to N^2(X)\) are adjoint, hence \(\lambda_1(\phi) = \lambda_2(\phi^{-1})\) and \(\lambda_2(\phi) = \lambda_1(\phi^{-1})\), and so 
\[
\lambda_1(\phi) = \lambda_2(\phi^{-1}) \geq \lambda_1(\phi^{-1})^2 = \lambda_2(\phi)^2 \geq \lambda_1(\phi)^4.
\]
By assumption \(\lambda_1(\phi) > 1\), so this is a contradiction. It must be that either \(D^2= 0\) or \((D^\prime)^2 = 0\). Replacing \(\phi\) by \(\phi^{-1}\) if needed, we can assume that \(D^2 = 0\).

Because the determinant of \(\phi^\ast\) is \(\pm 1\),  the only possible rational roots of the characteristic polynomial \(\det(\phi^\ast - \lambda I)\) are \(\lambda = 1\) or \(-1\).  The leading eigenvalue is a real number greater than \(1\), so it must be irrational, and the eigenvector \(D\) is not a multiple of a rational class.  

The same arguments apply to \(\phi^n\) for any positive integer \(n\), and so we may freely replace \(\phi\) by suitable iterate in later proofs and still assume that \(\lambda\) is irrational, while the eigenvector \(D\) remains unchanged.  In particular, \(\lambda^n\) is irrational for all nonzero \(n\).
\end{proof}
Observe that Lemma~\ref{numdim} makes use of the fact that \(\phi\) is an automorphism, and not merely a pseudoautomorphism: in the latter case, the action of \(\phi^\ast\) is not always compatible with intersections, and \(D^2\) is not necessarily an eigenvector. 

Taken together, Lemmas~\ref{intersections},~\ref{notpoint}, and~\ref{numdim}, show that Condition (A) is equivalent to the following condition, up to replacing \(\phi\) with \(\phi^{-1}\):
\begin{itemize}
\item[(A$^\prime$)] There exists a divisorial contraction \(\pi : X \to Y\) with \(Y\) smooth.  The exceptional divisor \(E\) is a ruled surface over a smooth curve, and \(E\) is not \(\phi\)-periodic.  There is a nef eigenvector \(D\) of \(\phi^\ast : N^1(X) \to N^1(X)\) with \(D^2 = 0\) and irrational eigenvalue \(\lambda\).
\end{itemize}

\begin{lemma}
\label{rationality}
Suppose that \(\phi : X \to X\) is an automorphism satisfying Condition ($A^\prime$).  After a suitable rescaling, the class \(D \vert_E\) is rational.  Moreover, \((D \cdot E \cdot E_n)_X = 0\) for every nonzero \(n\).
\end{lemma}
\begin{proof}
Since \(D\) is nef, so too is \(D\vert_E\).  We have \((D\vert_E \cdot D\vert_E)_E = (D \cdot D \cdot E)_X = 0\) by Lemma~\ref{numdim}, which shows that \(D\vert_E\) is not ample and hence lies on the boundary of \(\Nef(E)\).  The first claim is then a consequence of Proposition~\ref{ruledsurfacefacts}, because the nef cone of a ruled surface is bounded by rational classes.  We now assume that \(D\vert_E\) is rational.

For the second claim, we compute \((D \cdot E \cdot E_n)_X\) in two different ways:
\begin{align}
(D \cdot E \cdot E_n)_X &= (D\vert_E \cdot E_n \vert_E)_E \\ (D \cdot E \cdot E_n)_X &= \left( (\phi^\ast)^n(D) \cdot (\phi^\ast)^n(E) \cdot (\phi^\ast)^n(E_n)\right)_X \nonumber \\ 
  &= (\lambda^n D \cdot E_{-n} \cdot E)_X = \lambda^n(D \vert_E \cdot E_{-n}\vert_E)_E.
\end{align}
The right-hand side of (1) is the intersection of two \(\Q\)-Cartier divisors on a smooth surface, hence rational.  The right-hand side of (2) is an irrational multiple of a rational number.  The only possibility is that \((D \cdot E \cdot E_n)_X = 0\).
\end{proof}

\begin{lemma}
\label{hodgeindex}
Suppose that \(S\) is a smooth projective surface, and that \(D\) is a nonzero nef class in \(N^1(S)\).  The set of rays in \(N^1(S)\) represented by an irreducible curve \(C\) with \(D \cdot C = 0\) is finite.
\end{lemma}
\begin{proof}
It is an easy consequence of the Hodge index theorem that the number of irreducible curves with \(D \cdot C = 0\) and for which \([C]\) is not on the ray \(\Ras D\) is bounded by \(2 (\rho(S) - 2)\)~\cite[Lemma 3.1]{totaromoving}.  Together with the ray \(\Ras D\) itself, which may or may not be represented by a curve, this gives at most \(2 \rho(S) - 3\) rays in \(N^1(S)\) represented by curves with \(D \cdot C = 0\).
\end{proof}
We retain the notation that if \(g : S \to C\) is a ruled surface, with fiber \(f\), then \(\alpha\) is a generator of the bounding ray of \(\NEb(S)\) not spanned by \([f]\), and \(\beta\) is the a generator on the second bounding ray of \(\Nef(S)\).

\begin{lemma}
\label{exceptionalintersections} 
Suppose that \(\phi : X \to X\) is an automorphism satisfying Condition (A$^\prime$).  Then the ruled surface \(E\) is of type (R2) in the classification of Proposition~\ref{ruledsurfacefacts}, so that \(\alpha = \beta\).  The restriction \(D \vert_E\) is a nonzero multiple of \(\alpha\).  For every nonzero \(n\), the restriction \([E_n \vert_E]\) is a (possibly zero) multiple of \(\alpha\).
\end{lemma}
\begin{proof}
In light of Lemma~\ref{rationality}, in what follows we always assume that \(D\) is nef but not ample and is normalized so that \(D \vert_E\) is a rational class. There are three cases, which we treat separately: \(D \vert_E =  0\); \(D \vert_E\) is a nonzero multiple of a fiber \([f]\) of \(\pi\vert_E\); \(D \vert_E\) is a nonzero multiple of the nef boundary class \(\beta\).  We will see that the first two of these are impossible.

If \(D \vert_E = 0\), then \(D \cdot E = 0\) in \(N^2(X)\), and so \((\phi^\ast)^n(D \cdot E_n) = \lambda^n (D \cdot E) = 0\), which implies that \(D \vert_{E_n} = 0\) for any \(n\).  Let \(H \subset X\) be a very general member of a very ample linear system.  Each divisor \(E_n\) is contractible, and so is the unique effective divisor with class on the ray \(\Ras [E_n] \subset N^1(X)\).   The restriction map \(N^1(X) \to N^1(H)\) is injective by the Grothendieck-Lefschetz theorem, so the classes \(E_n\vert_H\) lie on distinct rays in \(N^1(H)\) as well.  As \(D\) is nef and nonzero and \(H\) is ample, \(D \vert_H\) is nef and nonzero, and we then compute
\[
(D\vert_H \cdot E_n\vert_H)_H = (D \cdot E_n \cdot H)_X = 0.
\]
This shows that \(D\vert_H\) vanishes on the infinitely many classes \(E_n\vert_H\).   By Bertini's theorem, since \(H\) is very general, each intersection \(E_n \cap H\) is an irreducible curve.  Since the rays \([E_n\vert_H]\) are distinct and represented by irreducible curves, this contradicts Lemma~\ref{hodgeindex}.

Suppose next that \(D \vert_E\) lies on \(\Ras [f]\).  The restriction \(E_1 \vert_E\) can be assumed nonzero by Lemma~\ref{intersections}, and \(E_1\vert_E\) is an effective class with \((D\vert_E \cdot E_1\vert_E)_E = 0\) by Lemma~\ref{rationality}.  It must be that \([E_1 \vert_E]\) lies on \(\Ras [f]\) as well. As before, 
\[
D \cdot E_1 = \phi_\ast((\phi^\ast)(D \cdot E_1)) = \phi_\ast(\phi^\ast D \cdot E) = \lambda \phi_\ast(D \cdot E).
\]
Since \(D \vert_E\) is on the ray \(\Ras [f]\), the restriction \(D \vert_{E_1}\) is on the ray \(\Ras [f_1] \subset N_1(E_1)\).  Then \((D\vert_E \cdot E_1 \vert_E)_E = (D \cdot E \cdot E_1)_X = 0\), so the restriction \([E \vert_{E_1}]\) must on the ray \(\Ras [f_1]\).  The only curves on a ruled surface numerically equivalent to a fiber are fibers, so any curve contained in \(E_1 \cap E\) is both a fiber of \(E\) and a fiber of \(E_1\).  But since the rays corresponding to the divisorial contractions of \(E\) and \(E_1\) are distinct, this is impossible.

The only remaining possibility is that \(D \vert_E\) is a nonzero multiple of \(\beta\).  For a ruled surface of type (R1), the class \(\beta\) has \(\beta^2 > 0\), but \((D\vert_E \cdot D\vert_E)_E = 0\), and so \(E\) must be of type (R2), with \(\alpha = \beta\).  At last, \(0 = (D \cdot E \cdot E_n)_X = (D\vert_E \cdot E_n\vert_E)_E\).  Since \(D\vert_E\) is proportional to \(\beta = \alpha\),  \([E_n\vert_E]\) must be a multiple of \(\alpha\) for every nonzero \(n\).
\end{proof}

In Example~\ref{productex}, we have \([E_1\vert_E] \in \Ras \alpha\). So we should not hope to prove that this subcase is impossible; instead we show that if \(X\) is of this type, the automorphism \(\phi : X \to X\) is imprimitive.  Indeed, in this and similar examples, the curves contained in \(E_1 \cap E\) are all fibers of a map to a surface.  We will show that this is always the case: if \(\xi\) is a curve in \(E_1 \cap E\), then there exists a rational fibration \(X \rat S\) with \(\xi\) as a fiber. 

\section{Some semilocal dynamics}
\label{localdynamics}

We now pause to prove some local dynamical results, dealing with the behavior of an automorphism \(\phi : X \to X\) in a formal neighborhood of a \(\phi\)-invariant curve, which is not necessarily fixed pointwise. 

\thmseparateiterates

\begin{example}
Consider the variety \(X = \P^2 \times \Gamma\), where \(\Gamma\) is an elliptic curve.  Let \(M : \P^2 \to \P^2\) be an infinite order automorphism of \(\P^2\) with isolated fixed points, and let \(\psi : \Gamma \to \Gamma\) be a non-torsion translation on \(\Gamma\), so that \(\phi = M \times \psi : X \to X\) is an infinite order automorphism.  The map \(M : \P^2 \to \P^2\) has at least one fixed point \(p\), and the curve \(C = p \times \Gamma\) is invariant under \(\phi\), but does not contain any fixed points.  If \(L \subset \P^2\) is a general line through \(p\), then \(E = L \times C\) is a divisor containing \(C\) which has infinite order under \(\phi\).  The divisors \(\phi^n(E)\) are all separated by the single blow-up \(\pi : \Bl_C X \to X\).
\end{example}

We first sketch the proof the two-dimensional analog of Theorem~\ref{separateiterates}, which suggests the strategy of the full proof.  A sharper two-dimensional statement in which it is not necessary to replace \(\phi\) by an iterate is due to Arnold, but the proof does not readily generalize to higher-dimensional settings in which \(\phi\) has no fixed points~\cite{arnold}.  The results of this section roughly extend Arnold's observation to threefolds, at the expense of requiring that \(\phi\) be replaced by an iterate.
\begin{lemma}
\label{surfacestatement}
Suppose that \(\phi : X \to X\) is an automorphism of a smooth projective surface, and that \(p\) is a fixed point of \(\phi\). Let \(C \subset X\) be a curve, smooth at \(p\) and with infinite order under \(\phi\).  After replacing \(\phi\) by some iterate, there exists a birational map \(\pi : Y \to X\) such that the strict transforms of the curves \(\phi^n(C)\) do not intersect above \(p\).
\end{lemma}
\begin{proof}
Choose local analytic coordinates \(x\) and \(y\) on a neighborhood of \(p\) so that \(C\) is defined by \(x = 0\).  Let \(\rho : \C[[x,y]] \to \C[[x,y]]\) be the pullback map induced by \(\phi\), so that \(\phi^n(C)\) is defined by \(\rho^n(x) = 0\).

Write \(M = \left( \begin{smallmatrix} a & c \\ b & d \end{smallmatrix} \right)\) for the linear part of \(\rho\) with respect to the basis given by \(x\) and \(y\). If \(\left( \begin{smallmatrix} 1 \\ 0 \end{smallmatrix} \right)\) is not an eigenvector of \(M^n\) for any \(n\), then the curves \(C_n = \phi^n(C)\) all have distinct tangent directions at \(p\), and blowing up the point \(p\) gives the resolution required by the lemma. If there is some \(n\) so that \(\left( \begin{smallmatrix} 1 \\ 0 \end{smallmatrix} \right)\) is an eigenvector, we may replace \(\phi\) by \(\phi^n\) and suppose that the coefficient \(b\) is \(0\), so that \(\rho(x)\) has no \(y^1\) term. Then \(C\) is tangent to \(\phi^n(C)\) at \(p\) for all \(n\).  Since \(M\) is invertible, \(a\) and \(d\) are both nonzero.

We now require an elementary observation on roots of unity.  Suppose that \(a\) and \(d\) are two nonzero complex numbers.  Then there exists an integer \(m\) such that for any positive integer \(k\) we have either: \((d^m)^k/(a^m)\) is not a root of unity, or \((d^m)^k/(a^m) = 1\).  Indeed, the subgroup \( \{ a^i d^j \} \cap \Omega\) of roots of unity of the form \(a^i d^j\) for integers \(i\) and \(j\) is a finitely generated subgroup of \(\Omega\), and finitely generated subgroups of \(\Omega\) are finite cyclic groups. If we replace \(a\) by \(a^m\) and \(d\) by \(d^m\), the corresponding subgroup is replaced by its \(m\)th power.  The claim follows by taking \(m\) to be divisible by the orders of all  elements of \(\set{a^i d^j} \cap \Omega\).

Replacing \(\phi\) by the iterate \(\phi^m\), we may then assume that for every value of \(k\), the quotient \(d^k/a\) is either not a root of unity, or is equal to \(1\). The curve \(C\) is not invariant under \(\phi\), so the function \(\rho(x)\) must have some term not divisible by \(x\); suppose the lowest-order such term is \(fy^k\), with nonzero \(f\) and \(k \geq 2\).  Then \(\rho\) descends to an automorphism of the two-dimensional vector space \(V = (x,y^k)/(x^2,xy,y^{k+1})\), a quotient of ideals in \(\C[[x,y]]\). Since in the quotient \(\rho(x) = ax+fy^k\) and \(\rho(y^k) = d^k y^k\), the matrix for the action of \(\rho\) on \(V\) with respect to the basis given by \(x\) and \(y^k\) is
\[
P = \begin{pmatrix} a & 0 \\ f & d^k \end{pmatrix} = a \begin{pmatrix} 1 & 0 \\ e & \delta \end{pmatrix},
\]   
where \(e = f/a\) and \(\delta = d^k/a\).  By assumption, the entry \(\delta\) is either \(1\) or is not a root of unity.  Then we have
\[
P^n = a^n \begin{pmatrix}
1 & 0 \\ \sum_{j=0}^{n-1} e \delta^j & \delta^n
\end{pmatrix} 
\]
If \(\delta\) is not equal to \(1\), the sum is computed as
\[
P^n = a^n \begin{pmatrix}
1 & 0 \\ \frac{1-\delta^n}{1-\delta}e & \delta^n.
\end{pmatrix}
\]
However, \(\delta\) is not a root of unity, so the factor \(\frac{1-\delta^n}{1-\delta}\) is nonzero, as is \(e\).  If \(\delta = 1\), then
\[
P^n = a^n \begin{pmatrix}
1 & 0 \\ ne & 1
\end{pmatrix}.
\]
Either way, we have verified that \(P^n\) has a nonzero entry in the lower left, so \(\phi^n(x)\) has nonzero \(y^k\) term for all \(n\).  Then a sequence of blow-ups at the point \(p\) separates all of the curves \(\phi^n(C)\).
\end{proof}

\begin{example}
\label{whyiterate}
To see why it is necessary to know that \(d^k/a\) is not a root of unity for any value of \(k\), consider the following simple example:
\begin{align*}
\rho(x) &= x + y^2 + f(y) \\
\rho(y) &= dy,
\end{align*}
We might hope to take \(k = 2\) in applying the argument to \(\rho\), so that \(\rho^n(x)\) has nonzero \(y^2\) term for all \(n\).  However, we have
\[
\rho^2(x) = (x+y^2+f(y)) + ((dy)^2 + f(dy)) =  x + (1+d^2)y^2 + \cdots
\]
If \(d^2 = -1\), then \(\rho^2(x)\) has no \(y^2\) term, and so the \(\rho^n(x)\) do not all have nonzero \(y^2\) term as needed.  We must replace \(\rho\) by \(\rho^2\) and try again.  After passing to this iterate, \(\rho(x)\) must again have some nonzero term \(fy^k\) with \(k \geq 3\), but it is difficult to control the value of \(k\) that occurs. If \(d^k/a\) is a root of unity (for the new value of \(k\)), some iterate will have vanishing \(y^k\) term, and it will be necessary to iterate a second time.  The observation on roots of unity shows that we can pass to a single fixed iterate, and that no matter what value of \(k\) appears in the leading \(y^k\) term of \(\rho(x)\), the ratio \(d^k/a\) is not a root of unity.
\end{example}

\begin{example}
The lemma is no longer true if \(\phi\) is not an automorphism: consider the map \(\phi : \P^1 \times \P^1 \to \P^1 \times \P^1\) defined on \(\A^{\!2}\) by \(\phi(x,y) = (x,y^2)\).  The curve \(C\) defined by \(y-x=0\) has \(\phi^n(C)\) defined by \(y - x^{2^n} = 0\). The curves \(\phi^n(C)\) and \(\phi^m(C)\) have unbounded orders of tangency at \((0,0)\) when \(m\) and \(n\) are both large, and there is no fixed blow-up on which these infinitely many curves are separated.
\end{example}

The strategy of the proof of Theorem~\ref{separateiterates} is similar: we consider the map \(\phi^\ast : \widehat{\cO}_{X,C} \to \widehat{\cO}_{X,C}\) induced on the completion of the local ring at \(C\). The proof is again a computation in coordinates, but this requires some care: although \(\widehat{\cO}_{X,C}\) is isomorphic (as a local ring) to a power series ring over the function field \(K(C)\), the pullback \(\phi^\ast : \widehat{\cO}_{X,C} \to \widehat{\cO}_{X,C}\) is not a map of \(K(C)\)-algebras.  A second difficulty is that the induced map \(\phi^\ast : K(C) \to K(C)\) on the residue field is not the identity.  As a result, we will see that when carrying out power series manipulations in \(\widehat{\cO}_{X,C}\), cancellations of coefficients as in Example~\ref{whyiterate} occur not only when \(d^k/a\) is a root of unity, but when \(d^k/a\) is of the form \(\omega \, f/\phi^\ast(f)\), where \(\omega\) is a root of unity and \(f\) is an element of \(K(C)\).  To address this difficulty, we must first prove some facts about the elements of \(K(C)\) of this form.

We begin with some definitions.  Suppose that \(k\) is an algebraically closed field of characteristic \(0\), and that \(K/k\) is an extension field.  Let \(r : K \to K\) be an automorphism of \(K\) fixing \(k\).  Given an element \(f\) of \(K\) and a non-negative integer \(n\), define
\begin{align*}
\tau_r(f) &= f/r(f) \\
\alpha_r(f,n) &= f \, r(f) \cdots r^{n-1}(f) 
\end{align*}
Here \(\alpha_r(f,n)\) is defined for any non-negative integer \(n\), with \(\alpha_r(f,0) = 1\).  Both \(\tau_r(-)\) and \(\alpha_r(-,n)\) define multiplicative homomorphisms \(K^\times \to K^\times\).

The next lemma collects some additional identities satisfied by these functions, which will simplify some of the upcoming calculations.
\begin{lemma}
\label{someidentities}
Suppose that \(f \in K\), \(c \in k\), and \(n \in \Z_{\geq 0}\).  Then \(\tau_r\) and \(\alpha_r\) satisfy the following identities.
\begin{enumerate}
\item \(\alpha_r(cf,n) = c^n \alpha_r(f,n)\)
\item \(\alpha_r(\tau_r(f),n) = f/r^n(f) = \tau_{r^n}(f)\)
\item \(f \, r(\alpha_r(f,n)) = \alpha_r(f,n+1) = \alpha_r(f,n) r^n(f)\)
\item \(\alpha_r(\alpha_{r^m}(f,n),m) = \alpha_r(f,mn)\)
\end{enumerate}
\end{lemma}
\begin{proof}
For (1),
\[
\alpha_r(cf,n) = (cf) \, r(cf) \cdots r^{n-1}(cf) = c^{n} f \, r(f) \cdots r^{n-1}(f) = c^n \alpha(f,n).
\]
For (2), we have
\[
\alpha_r(\tau_r(f),n) = \frac{f}{r(f)} \frac{r(f)}{r^2(f)} \cdot \frac{r^{n-1}(f)}{r^n(f)} = \frac{f}{r^n(f)} = \tau_{r^n}(f).
\]
For (3), simply note that
\begin{align*}
f \, r(\alpha_r(f,n)) &= f \, r(f \, r(f) \cdots r^{n-1}(f)) = f \, r(f) \, r^2(f) \cdots r^{n}(f) \\ &= \alpha_r(f,n+1) = \alpha_r(f,n) r^n(f).
\end{align*}
The last claim (4) is checked by
\begin{align*}
\alpha_r(\alpha_{r^m}(f,n),m) &= \alpha_r(f \, r^m(f) \cdots r^{m(n-1)}(f),m) \\
&= \left( \alpha_r(f,m) \right)  \left( \alpha_r(r^m(f),m) \right)  \cdots  \left( \alpha_r(r^{m(n-1)}(f,m)) \right)  \\
&= \left( f \, r(f) \cdots r^{m-1}(f) \right) \left( r^m(f) \cdots r^{2m-1}(f) \right) \cdots \left( r^{m(n-1)+1}(f) \cdots f^{mn-1}(f) \right) \\ &= \alpha_r(f,mn).\qedhere
\end{align*}
\end{proof}
\noindent We say that \(r : K \to K\) is \emph{shifting} over \(k\) if for any \(f\) in \(K\):
\begin{enumerate}
\item[(S1)] If \(\alpha_r(f,n) = 1\) for some \(n \geq 1\), then \(f\) is an \(n\)\nth{} root of unity in \(k\).
\item[(S2)] If \(\tau_r(f)\) is a root of unity, then \(\tau_r(f) = 1\).
\end{enumerate}
Consider also the related condition
\begin{enumerate}
\item[(S1$^\prime$)] If \(f\) is an element of \(K\) with \(f = r^{n}(f)\) for some \(n \geq 1\), then \(f\) lies in \(k\).
\end{enumerate}
Suppose that \((R,\fm)\) is a local \(k\)-algebra with residue field \(K = R/\fm\), and that \(r : K \to K\) is a shifting automorphism. We say that a local \(k\)-algebra automorphism \(\rho : R \to R\) is \(r\)-shifting if the induced map on the residue field coincides with \(r : K \to K\).  The next lemma collects a few elementary observations about \(r\)-shifting automorphisms.
\begin{lemma}
\label{shiftingproperties}
Suppose that \(r  : K \to K\) is shifting, and \(\rho : R \to R\) is an \(r\)-shifting automorphism of a local ring.
\begin{enumerate}
\item If \(r\) satisfies condition (S1$^\prime$), then \(r\) satisfies condition (S1).
\item If \(\alpha_r(f,m)\) is an \(n\)\nth root of unity, then \(f\) is an \(mn\)\nth root of unity.
\item The iterate \(r^m : K \to K\) is shifting for any integer \(m \geq 1\).
\item The iterate \(\rho^m : R \to R\) is \(r^m\)-shifting for any \(m \geq 1\).
\end{enumerate}
\end{lemma}
\begin{proof}

First we check (1).  Suppose that \(r\) satisfies condition (S1$^\prime$).   By (3) of Lemma~\ref{someidentities}, we have \(f \, r(\alpha_r(f,n)) = \alpha_r(f,n+1) = \alpha_r(f,n) r^{n}(f)\).  If  \(\alpha_r(f,n) = 1\), then \(r(\alpha_r(f,n)) = 1\) as well, and so \(f = r^{n}(f)\).  By (S1$^\prime$), we have \(f \in k\), and so \(1 = \alpha_r(f,n) = f^n\) and \(f\) must be an \(n\)\nth root of unity.

Next we prove (2).  Suppose that \(\alpha_r(f,m) = \omega_n\).  Let \(\zeta\) be an \(m\)\nth root of \(\omega_n\) in \(k\).  Then \(\alpha_r(\zeta,m) = \zeta^m = \omega_n\), and so \(\alpha_r(f/\zeta,m) = 1\).  By (S1), it must be that \(f/\zeta\) is an \(m\)\nth root of unity.  Since \(\zeta\) is an \(mn\)\nth root of unity, \(f\) is itself an \(mn\)\nth root of unity.

To prove (3), we first check condition (S1) for \(r^m\).  Suppose that \(\alpha_{r^m}(f,n) = 1\). Then \(\alpha_r(f,mn) = \alpha_r(\alpha_{r^m}(f,n),m) = \alpha_r(1,m) = 1\).  By condition (S1) for \(r\), \(f\) is an \(mn\)\nth root of unity in \(k\).
That \(f\) is in \(k\) implies that  \(\alpha_{r^m}(f,n) = f^n = 1\), and \(f\) is in fact an \(n\)\nth root of unity as needed. Suppose now that \(\tau_{r^m}(f) = \omega_n\) is an \(n\)\nth root of unity.  Then \(\alpha_r(\tau_r(f),m) = \tau_{r^m}(f) = \omega_n\) by (2) of Lemma~\ref{someidentities}.
By (2) above, \(\tau_r(f)\) is an \(mn\)\nth root of unity.  But Condition (S2) for \(r\) implies that \(\tau_r(f) = 1\).  Then \(\tau_{r^m}(f) = \alpha_r(\tau_r(f),m) = 1\), as required for (S2).

Claim (4) is immediate from the definition. 
\end{proof}

\begin{lemma}
\label{curveshifting}
Suppose that \(\phi : C \to C\) is an automorphism of an integral curve over \(k\).  Let \(r : K \to K\) be the pullback map on the function field \(K = K(C)\). Then some iterate of \(r\) is shifting.
\end{lemma}
\begin{proof}
If \(\phi\) has finite order, then some iterate of \(r\) is the identity, which is trivially shifting.  If \(\phi\) has infinite order, some point \(z \in C\) has infinite, hence Zariski dense, orbit.  Suppose that \(f\) is an element of \(K\) with \(f = r^n(f)\) for some nonzero \(n\), so that \(f = r^{mn}(f)\) for any integer \(m\).  Then \(f(z) = f(\phi^{mn}(z))\) for all \(m\), and \(f\) must be constant.  Consequently \(r\) satisfies condition (S1$^\prime$) and condition (S1).

Suppose that \(f/r(f) = \omega_n\) is a root of unity.  Then \(f(\phi^m(z)) = \omega_n^m f(z)\), so there is a Zariski dense set of points \(\phi^m(z)\) with \(f(\phi^m(z))\) an \(n\)\nth root of unity.  Then \(f\) must be constant, and so \(f/r(f) = 1\).
\end{proof}

Say that \(f \in K\) is an \(r\)-coboundary if \(f = \tau_r(g)\) for some \(g\).  Similarly, say that \(f\) is an \(n\)\nth \(r\)-root of unity if \(f = \omega_n \, \tau_r(g)\), where \(\omega_n \in k\) is an \(n\)\nth root of unity and \(g \in K\).  We now collect some simple observations about \(r\)-coboundaries and \(r\)-roots of unity, generalizing properties of the roots of unity in \(k\).

\begin{lemma}
\label{quotientsandroots}
Suppose that \(f\) is an element of \(K\).
\begin{enumerate}
\item The \(r\)-roots of unity and \(r\)-coboundaries are multiplicative subgroups of \(K^\times\).
\item  If \(f\) is an \(n\)\nth \(r\)-root of unity for some \(n\), then \(\alpha_r(f,n)\) is an \(r^n\)-coboundary.  
\item If \(\alpha_r(f,n) = \omega_k \tau_{r^n}(g)\) is an \(r^n\)-root of unity, then \(f = \zeta \, \tau_r(g)\) is an \(r\)-root of unity.
\item Suppose that \(f\) and \(g\) are two elements of \(K\).  There exists \(m\) such that for all \(k\), either \(\alpha_r(f^k/g,m)\) is an \(r^m\)-coboundary, or \(\alpha_r(f^k/g,m)\) is not an \(r^m\)-root of unity.
\end{enumerate}
\end{lemma}
\begin{proof}
Statement (1) is clear. For (2), if \( f = \omega_n \tau_r(g)\), then
\[
\alpha_r(f,n) = \alpha_r(\omega_n \tau_r(g),n) = \omega_n^n \alpha_r(\tau_r(g),n) = \tau_{r^n}(g).
\]
On the other hand, for (3), suppose that \(\alpha_r(f,n)\) is a \(k\)\nth \(r^n\)-root of unity, so that \(\alpha_r(f,n) = \omega_k \tau_{r^n}(g)\).  Let \(f_0 = \zeta \, \tau_r(g)\), where \(\zeta\) is chosen so that \(\zeta^n = \omega_k\).  We have
\[
\alpha_r(f_0,n) = \alpha_r(\zeta \tau_r(g),n) = \zeta^n \alpha_r(\tau_r(g),n) = \omega_k \tau_{r^n}(g) = \alpha_r(f,n),
\]
and so \(\alpha_r(f/f_0,n) = 1\).  By (S1), we have \(f/f_0 = \omega_n\), which gives \(f = \omega_n f_0 = (\omega_n \zeta) \tau_r(g)\), and we conclude that \(f\) is an \(r\)-root of unity.

The set \(\Sigma_f = \set{n : \text{$f^n$ is an $r$-root of unity}}\) is evidently a subgroup of \(\Z\).  Let \(\Sigma_{f,g} = \set{n : \text{$f^n/g$ is an $r$-root of unity}}\).  If \(k\) and \(\ell\) both lie in \(\Sigma_{f,g}\), then 
we can write \(f^k/g = \omega_m \tau_r(a)\) and \(f^\ell/g = \omega_n \tau_r(b)\) and so
\[
f^{k-\ell} = \frac{f^k/g}{f^\ell/g} = \frac{\omega_m}{\omega_n} \tau_r(a/b).
\]
Thus \(f^{k-\ell}\) is an \(r\)-root of unity and \(k-\ell\) is a member of \(\Sigma_f\). It follows that \(\Sigma_{f,g}\) is a coset of \(\Sigma_f\) in \(\Z\).

Suppose first that no nonzero power of \(f\) is an \(r\)-root of unity, so \(\Sigma_f = \set{0}\) is trivial.  Then there is at most one value \(k\) such that \(f^k/g\) is an \(r\)-root of unity; write \(f^k/g = \omega_m \tau_r(h)\).  We claim that (4) holds for this value of \(m\).  Indeed, since \(f^k/g\) is an \(m\)\nth \(r\)-root of unity, by (1) \(\alpha_r(f^k/g,m)\) is an \(r^m\)-coboundary.  On the other hand, if \(\ell \neq k\), then \(f^\ell/g\) is not an \(r\)-root of unity, and by claim (2), \(\alpha_r(f^\ell/g,m)\) is not an \(r^m\)-root of unity.

Suppose instead that \(\Sigma_f\) is infinite, and let \(e \in \Sigma_f\) be the positive generator.  Write \(f^e = \omega_m \tau_r(a)\).  Fix some \(k_0\) in \(\Sigma_{f,g}\) and write \(f^{k_0}/g = \omega_n \tau_r(b)\).  If \(k\) is an any element of \(\Sigma_{f,g}\), we have \(k = e \ell + k_0\) for some \(\ell\).  Then
\[
\frac{f^k}{g} = \frac{f^{e \ell +k_0}}{g} = (f^e)^\ell \frac{f^{k_0}}{g} = (\omega_m \tau_r(a))^\ell (\omega_n \tau_r(b)) = (\omega_m \omega_n) \tau_r(a^\ell b).
\]
This shows that \(f^k/g\) is an \(mn\)\nth \(r\)-root of unity, independent of \(k \in \Sigma_{f,g}\).  By observation (1), if \(k \in \Sigma_{f,g}\), then \(\alpha_r(f^k/g,mn)\) is an \(r^{mn}\)-coboundary.  On the other hand, if \(k \not\in \Sigma_{f,g}\), then \(f^k/g\) is not an \(r\)-root of unity, and so \(\alpha_r(f^k/g,mn)\) is not an \(r^{mn}\)-root of unity by claim (3).  This completes the proof of (4).
\end{proof}

\begin{lemma}
\label{shiftinglemma}
Let \(k\) be an algebraically closed field of characteristic \(0\).  Suppose that  \(R\) is a regular local \(k\)-algebra of dimension \(2\) with maximal ideal \(\fm\) and residue field \(R/\fm \cong K\), and that \(\rho : R \to R\) is a local \(k\)-algebra automorphism inducing a shifting automorphism \(r : K \to K\).  Let \(I\) be a height-\(1\) prime ideal not contained in \(\fm^2\).  After replacing \(\rho\) by a suitable iterate, there exists an integer \(N\) such that for any nonzero \(n\), we have \(I + \rho^n(I) = I + \fm^N\).  
\end{lemma}
\begin{proof}
Let \(\wR\) be the completion of \(R\) along \(\fm\).  Suppose that \(J \subset \wR\) is an ideal with \(\rho(J) = J\).  We will write \(\rho_J : J \to J\) for the restriction of \(\rho\) to \(J\) (as a map of \(\wR\)-modules), and \(\sigma_J : J/\fm J \to J/\fm J\) for the induced map on the quotient by the maximal ideal.  For any such \(J\), the map \(\rho_J\) is a \(\rho\)-semilinear map of \(\wR\)-modules, in the sense that if \(f \in \wR\) and \(j \in J\), we have \(\rho_J(fj) = \rho(f) \, \rho_J(j)\).  The induced \(\sigma_J : J/\fm J \to J/\fm J\) is an \(r\)-semilinear map of \(K = R/\fm\)-modules, so that if \(f \in K\) and \(j \in J/\fm J\), we have \(\sigma_J(fj) = r(f) \sigma_J(j)\).

Since \(R\) is a regular local ring, it is a UFD and every height-\(1\) prime ideal is principal.  In particular, we can take \(I=(x)\) with \(x\) not an element of \(\fm^2\).  By the Cohen structure theorem, there exists some \(y\) in \(\wR\) such that \(\wR \cong K[[x,y]]\)~\cite[Proposition 10.16]{eisenbud}.   This isomorphism is not canonical: for example, \(\rho : K[[x,y]] \to K[[x,y]]\) is not necessarily a map of \(K\)-algebras, and the coefficient field \(K \subset K[[x,y]]\) may not be fixed by \(\rho\).  It will nevertheless be convenient to work with \(\wR\) as a power series ring: we fix some such \(y\) and an identification \(\wR \to K[[x,y]]\).  Observe that any prime ideal containing \(I\) must be either \(I\) itself or of the form \(I + \fm^k\); in particular, each of the ideals \(I+ \rho^n(I)\) we are considering is of this form for some \(k\).

Write \(\rho(x) = \sum_{i,j} a_{ij} x^i y^j\) and \(\rho(y) = \sum_{i,j} b_{ij} x^i y^j\), with \(a_{ij}\) and \(b_{ij}\) in the coefficient field \(K \subset \wR\) (note, however, that \(\rho(a_{ij})\) and \(\rho(b_{ij})\) are not necessarily elements of \(K\)).   We first consider the linear part \(\sigma_\fm : \fm/\fm^2 \to \fm/\fm^2\).   By semilinearity, if \(c\) and \(d\) are any elements of \(R/\fm\), we have
\begin{align*}
  \sigma_{\fm}(cx+dy) &= r(c) \, \sigma_\fm(x) + r(d) \, \sigma_{\fm}(y) \\
&= r(c) (a_{10} x + a_{01} y ) + r(d) (b_{10} x + b_{01} y) \\
&= (a_{10} \, r(c) + b_{10} \, r(d)) x + (a_{01} \, r(c) + b_{01} \, r(d)) y,
\end{align*}
which shows that if \(v \in \fm\) is regarded as a vector with respect to the basis given by \(x\) and \(y\), we have \(\sigma_{\fm}(v) = M \, r(v)\), where \(M = \left( \begin{smallmatrix} a_{10} & a_{01} \\ b_{10} & b_{01} \end{smallmatrix} \right)\).  Iterating gives
\begin{align*}
\sigma_{\fm}^n(v) &= M \, r(M \, r(\cdots M \, r(v))) \\
  &= M \, r(M) \cdots \, r^{n-1}(M) r^n(v) = \alpha_r(M,n)(r^n(v)),
\end{align*}
where \(r\) acts on matrices entrywise.  In particular, \(\sigma_{\fm}^n(x) = \alpha_r(M,n)(x)\).  If there is no \(n > 0\) for which \((1,0)\) is an eigenvector of of \(\alpha_r(M,n)\), then \(\sigma_{\fm}^n(x)\) has a nonzero \(y\) component for all nonzero \(n\). In this case, \(\rho^n(x)\) also has a nonzero \(y\) component for nonzero \(n\), and the the lemma holds with \(N = 1\).

If there is a value of \(n\) for which \(\sigma_{\fm}^n(x)\) has zero coefficient on \(y\), we can replace \(\rho\) by \(\rho^n\) and then assume that \(a_{01} = 0\).  The fact that \(\rho\) is an automorphism implies that \(M\) is invertible, so that \(a_{10}\) and \(b_{01}\) are both nonzero.  Since \(M\) is upper triangular, so too is \(\alpha_r(M,n)\) for any \(n\).  If we replace \(\rho\) by \(\rho^n\), then in the linear term, \(a_{10}\) is replaced by \(\alpha_r(a_{10},n)\) and \(b_{01}\) is replaced by \(\alpha_r(b_{01},n)\).  By Lemma~\ref{quotientsandroots}(4), applied to \(b_{01}\) and \(a_{10}\), there exists some \(m\) such that for every \(k\), either \(\alpha_r(b_{01}^k/a_{10},m)\) is an \(r^m\)-coboundary, or \(\alpha_r(b_{01}^k/a_{10},m)\) is not an \(r^m\)-root of unity.   Replacing \(\rho\) by \(\rho^m\), we may then assume that for every \(k\), if \(b_{01}^k/a_{10}\) is an \(r\)-root of unity, then \(b_{01}^k/a_{10}\) is an \(r\)-coboundary.

Because \(I=(x)\) is not invariant under \(\rho\), it must be that \(\rho(x)\) is not contained in \(I\). Let \(a_{0k} y^k\) be the lowest order nonzero term in \(\rho(x)\) that is not divisible by \(x\).  Because \(a_{01} = 0\), we have \(k \geq 2\).  

Let \(J\) be the ideal \((x,y^k) \subset \wR\).  We have \(\rho(y) \in \fm\), so \(\rho(y^k) \in \fm^k \subset J\).  Similarly, \(\rho(x) \in (x,y^k) = J\), and so \(\rho(J) = J\).  We now consider the map \(\sigma_J : J/\fm J \to J/\fm J\).  This is an \(r\)-semilinear map of \(K\)-vector spaces, so that if \(a \in K\) and \(v \in J/\fm J\), we have \(\sigma_J(av) = r(a) \, \sigma_J(v)\).   Since \(\sigma_J(x) = a_{10} x + a_{0k} y^k\) and \(\sigma_J(y^k) = b_{01}^k y^k\), the matrix for \(\sigma_J\) with respect to the basis given by \(x\) and \(y^k\) is
\[
P = \begin{pmatrix}
a_{10} & 0 \\
a_{0k} & b_{01}^k
\end{pmatrix}.
\]
The map \(\sigma_J\) is \(r\)-semilinear, so the matrix for the action of \(\sigma_J^n\) is \(P \, r(P) \cdots r^{n-1}(P) = \alpha_r(P,n)\).  For readability, set \(e = a_{0k}/a_{10}\) and \(\delta = b_{01}^k/a_{10}\), and consider the matrix 
\[
P_1 = a_{10}^{-1} P = 
\begin{pmatrix}
1 & 0 \\
e & \delta
\end{pmatrix}.
\]

Then \(\alpha_r(P_1,n) = \alpha_r(a_{10}^{-1},n) \alpha_r(P,n)\).  By the above reduction, if \(\delta\) is an \(r\)-root of unity, it is an \(r\)-coboundary.  We claim next that by \(r\)-semilinearity of \(\sigma_J\), for any \(n\) the matrix for \(\alpha_r(P_1,n)\) with respect to the basis given by \(x\) and \(y^k\) is given by
\[
\alpha_r(P_1,n) = \begin{pmatrix}
1 & 0 \\
\sum_{i=0}^{n-1}  \alpha_r(\delta,i) r^i(e) & \alpha_r(\delta,n)
\end{pmatrix}
\]
This is correct for \(n=1\), while we find
\begin{align*}
\alpha_r(P_1,n+1) &= \alpha_r(P_1,n) r^n(P_1) = 
\begin{pmatrix}
1 & 0 \\
\sum_{i=0}^{n-1} \alpha_r(\delta,i)  r^i(e) & \alpha_r(\delta,n)
\end{pmatrix}
\begin{pmatrix}
1 & 0 \\
r^n(e) & r^n(\delta)
\end{pmatrix} \\
&= \begin{pmatrix}
1 & 0 \\
\sum_{i=0}^{n-1}  \alpha_r(\delta,i) r^i(e) + \alpha_r(\delta,n)  r^n(e) & \alpha_r(\delta,n) r^n(\delta)
\end{pmatrix} \\
&= \begin{pmatrix}
1 & 0 \\
\sum_{i=0}^{n} \alpha_r(\delta,i)  r^i(e) & \alpha_r(\delta,n+1)
\end{pmatrix},
\end{align*}
as required. We will show that the lower-left entry of this matrix is non-zero for all nonzero \(n\), so that \(\sigma_J^n(x)\) (and hence \(\rho^n(x)\)) has nonzero coefficient on the term \(y^k\). Let \(S_n = \sum_{i=0}^{n-1} \alpha_r(\delta,i)  r^i(e)\).  Then
\begin{align*}
e + \delta \, r(S_n) &= e + \delta \, r\left( \sum_{i=0}^{n-1} \alpha_r(\delta,i)  r^{i}(e) \right)  = e + \delta \sum_{i=0}^{n-1} \frac{\alpha_r(\delta,i+1)}{\delta}  r^{i+1}(e) \\
&= e + \sum_{i=1}^{n} \alpha_r(\delta,i)  r^i(e)= \sum_{i=0}^{n} \alpha_r(\delta,i)  r^i(e) = S_n + \alpha_r(\delta,n) r^n(e).
\end{align*}

Suppose that \(S_n = 0\). Then \(\delta \, r(S_n) = 0\), and this implies that \(e = \alpha_r(\delta,n) r^n(e)\), whence \(\alpha_r(\delta,n) = e/r^n(e) = \tau_{r^n}(e)\).  Lemma~\ref{quotientsandroots}(3) implies that \(\delta\) is an \(r\)-root of unity, and in fact that \(\delta = \omega_n \tau_r(e)\) for some \(\omega_n\).  However, we have reduced to the case that if \(\delta\) is an \(r\)-root of unity, then \(\delta\) is an \(r\)-coboundary, and so it must be that \(\delta = \tau_r(h)\) for some \(h\). This gives \(\tau_r(h) = \omega_n \tau_r(e)\), and so \(\tau_r(h/e) = \omega_n\). By shifting hypothesis (S2), we have \(\omega_n = 1\), and so in fact \(\delta = \tau_r(e)\).  At last, we compute
\[
S_n = \sum_{i=0}^{n-1} \alpha_r(\delta,i) r^i(e) = \sum_{i=0}^{n-1} \alpha_r(\tau(e),i) r^i(e) = \sum_{i=0}^{n-1} \tau_{r^i}(e) r^i(e) = \sum_{i=0}^{n-1} \frac{e}{r^i(e)} r^i(e) = ne,
\]
which is nonzero because \(K\) has characteristic \(0\) and \(e \neq 0\).  Consequently \(S_n\) cannot be \(0\), which shows that \(\alpha_r(P_1,n)\), and thus \(\alpha_r(P,n)\), has nonzero entry in the lower-left.  In other words, for any positive integer \(n\), the iterate \(\sigma_J^n(x)\) has a nonzero coefficient on the \(y^k\) term. 

The claim of the theorem now holds with \(N = k\).  Since \(\rho^n(x)\) has no terms of pure \(y\) of degree lower than \(y^k\), and \(\rho^n(I)\) is the principal ideal generated by \(\rho^n(x)\), we have \(I + \rho^n(I) \subseteq I + \fm^N\).  On the other hand,  the coefficient on \(y^k\) in \(\rho^n(x)\) is nonzero, so \(\fm^N \subseteq I + \rho^n(I)\).  

The claim for negative \(n\) follows by the same argument. The analog of the matrix \(P\) for \(\rho^{-1}\) is \(r^{-1}(P^{-1})\), and the corresponding value of \(\delta\) is \(\delta^\prime = r^{-1}(\delta^{-1})\).  If \(\delta^\prime = \omega \, \tau(g)\) is an \(r\)-root of unity, then \(\delta = \omega \tau(r(g^{-1}))\), and so \(\omega = 1\).  Thus \(\delta^\prime\) is either not an \(r\)-root of unity, or is an \(r\)-coboundary, and the same argument applies with the same value of \(N\).  

The proves that \(I  + \rho^n(I) = I + \fm^N\) for all nonzero \(n\).  The corresponding equality of ideals in the non-completed ring \(R\) is immediate from faithful flatness of \(R \to \wR\).
\end{proof}

\begin{proof}[Proof of Theorem~\ref{separateiterates}]
Suppose that \(\phi : X \to X\) is as in the statement. Then \(\phi\) induces an automorphism of the local ring \(\phi^\ast : \cO_{X,C} \to \cO_{X,C}\).  Because \(C\) is smooth, \(\cO_{X,C}\) is a regular local ring, with residue field \(K(C)\).  After replacing \(\phi\) by an iterate, the induced automorphism \(r : K(C) \to K(C)\) can be assumed to be shifting by Lemma~\ref{curveshifting}, and the map \(\phi^\ast : \cO_{X,C} \to \cO_{X,C}\) is \(r\)-shifting.  Let \(\fm\) be the maximal ideal in \(\cO_{X,C}\), and let \(I \subset \cO_{X,C}\) the ideal defined by \(E\).  As \(E\) is a divisor smooth at the generic point of \(C\), \(I\) is a height-\(1\) prime contained in \(\fm\) and not containing \(\fm^2\). The divisor \(E\) is assumed to have infinite orbit under \(\phi\) and is irreducible, so the ideal \(I\) has infinite order under \(\phi^\ast\).  By Lemma~\ref{shiftinglemma}, there exists \(N\) such that \(I + \rho^n(I) = I + \fm^N\) for all nonzero values of \(n\). 

We now realize \(Y\) by a sequence of smooth blow-ups centered above \(C\).  Let \(\cI \subset \cO_X\) be the ideal sheaf of \(E\), and let \(\fn \subset \cO_X\) be the ideal sheaf of \(C\). The restriction of \(\cI\) to the stalk at the generic point of \(C\) is \(I \subset \cO_{X,C}\), while the restriction of \(\fn\) to this stalk is the maximal ideal \(\fm \subset \cO_{X,C}\).  For \(1 \leq i \leq N\) define \(\cI^{\leq i} = \prod_{k=1}^i \cI+\fn^k\).

Let \(X_0 = X\), and for \(1 \leq i \leq N\) let \(\sigma_i : X_i \to X\) be the blow-up of \(X\) along the ideal sheaf \(\cI^{\leq i}\).  Since \(\cI^{\leq i+1} = \cI^{\leq i} \cdot (\cI+\fm^{i+1})\), there is an induced morphism \(\pi_i : X_{i+1} \to X_i\).  Indeed, \(X_{i+1} \to X_i\) is the blow-up of \(X_i\) along the curve \(\pi_i^{-1}(C) \cap E\), and so \(X_{i}\) is smooth for every value of \(i\).  For example, \(\pi_1 : X_1 \to X\) is the blow-up of \(X\) along \(C\), and \(\pi_2 : X_2 \to X_1\) is the blow-up along \(E \cap F_1\), where \(F_1\) is the exceptional divisor of \(\pi_1\).

We claim that \(\phi\) lifts to an automorphism of \(X_i\) for every \(1 \leq i \leq N\).  Indeed, in the local ring \(\cO_{X,C}\) we have \(\rho(I + \fm^i) = \rho(I) + \rho(\fm^i) \subseteq (I+\fm^N) + \fm^i \subseteq I+\fm^i\).  The reverse inclusion follows from the same argument, and so \(\rho(I + \fm^i) = I + \fm^i\).  This holds at the generic point of \(C\), so it holds on some open set \(U \subset X\), and because neither \(\rho(I+\fm^i)\) or \(I + \fm^i\) has embedded points on \(C\) (except the generic point), this yields \(\phi^\ast(\cI+\fn^i) = \cI+\fn^i\) as ideal sheaves in \(\cO_X\).
It follows that \(\rho(\cI^{\leq i}) = \cI^{\leq i}\), and \(\phi\) lifts to an automorphism of \(X_i\).  Take \(Y = X_N\), so that \(\phi\) lifts to an automorphism of \(Y\) and condition (1) of the theorem is satisfied.  Note that the cosupport of \(\cI^{\leq i}\) is equal to \(C\), and so \(\pi : Y \to X\) is an isomorphism away from \(C\), as required by (2).

For every \(n > 0\), we have \(I+\rho^n(I) = I + \fm^{N}\) in the local ring \(\cO_{X,C}\).  Thus for each \(n\), there is an open set \(U_n \subset X\), containing the generic point of \(C\), such that \((\cI + \rho^n(\cI))\vert_{U_n}=  (\cI + \fn^{N})\vert_{U_n}\).  Let \(\pi_0 : Y_0 \to X\) be the blow-up of \(X\) along \(\cI+\fn^N\).    By \cite[Ch.\ II, Exercise 7.12]{hartshorne}, the strict transforms of \(E \vert_{U_n}\) and \(E_n \vert_{U_n}\) are disjoint in \(\pi_0^{-1}(U_n) \subseteq Y_0\). Since \(\cI^{\leq N} = \cI^{\leq N-1} \cdot (\cI + \fn^N)\), the map \(\pi : Y \to X\) factors through \(Y_0\), and so the corresponding strict transforms are disjoint in \(\pi^{-1}(U_n)\) as well.

This shows that \(\pi(E \cap E_n)\) does not contain all of \(C\).  For the intersections \(E_m \cap E_n\), we use the fact that \(\phi\) lifts to an automorphism \(\psi : Y \to Y\) to conclude that
\[
\pi(E_m \cap E_n) = \pi(\psi^m(E \cap E_{n-m})) = \phi^m(\pi(E \cap E_{n-m})).
\]
The intersection \(\pi(E \cap E_{n-m})\) does not contain all of \(C\), and \(C\) is \(\phi\)-invariant, we conclude point (3).  Note that for each \(m\) and \(n\), the image \(\pi(E_m \cap E_n)\) is at most a finite set of points, and so over a very general point of \(C\), the strict transforms of the divisors \(E_m\) and \(E_n\) are disjoint.
\end{proof}

\section{Construction of an equivariant fibration}

Suppose now that \(\phi : X \to X\) is an automorphism satisfying Condition (A$^\prime$).  We next apply Theorem~\ref{separateiterates} to pass to a birational model on which there are no \(E\)-rigid, \(\phi\)-periodic curves.

\begin{lemma}
\label{killperiodiccurves}
Suppose that  \(\phi : X \to X\) is an automorphism satisfying Condition (A$^\prime$).  After replacing \(\phi\) by an iterate, there exists a birational model \(\pi : Y \to X\) with the following properties:
\begin{enumerate}
\item The map \(\phi\) lifts to an automorphism \(\psi : Y \to Y\).
\item Write \(F\) for the strict transform of \(E\) on \(Y\).  No \(F\)-rigid curve is \(\psi\)-periodic.
\item Let \(g \subset F\) the strict transform of a general fiber \(f \subset E\).  Then \(F \cdot g < 0\).
\end{enumerate}
\end{lemma}
\begin{proof}
We will construct a sequence of birational models \(X_i\) such that \(\phi\) lifts to an automorphism \(\phi_i : X_i \to X_i\).  Let \(X_0 = X\) and \(\phi_0 = \phi\). On each model \(X_i\), we will write \(F_i\) for the strict transform of the divisor \(E = F_0 \subset X_0\).   Let \(N(\phi_i)\) be the number of \(\phi_i\)-periodic \(F_i\)-rigid curves. This number is certainly finite, for there are only finitely many \(F_i\)-rigid curves.  Suppose that \(\xi \subset F_i\) is an \(F_i\)-rigid \(\phi_i\)-periodic curve.  Replacing \(\phi_i\) by a suitable iterate, we may assume that \(\xi\) is fixed by \(\phi_i\). Passing to an iterate does not change the set of periodic curves.

By Theorem~\ref{separateiterates}, there exists a birational map \(\pi_i : X_{i+1} \to X_i\) with the property that \(\phi_i\) lifts to an automorphism \(\phi_{i+1} : X_{i+1} \to X_{i+1}\), and such that \(\pi_i(F_{i+1} \cap \phi_{i+1}^n(F_{i+1}))\) does not contain \(\xi\) for any nonzero \(n\).
\[
\xymatrix{
X_{i+1} \ar[r]^{\phi_{i+1}} \ar[d]_{\pi_i} & X_{i+1} \ar[d]_{\pi_i} \\
X_i \ar[r]^{\phi_i} & X_i
}
\]

The map \(\pi_i \vert_{F_{i+1}} : F_{i+1} \to F_i\) is an isomorphism. Let \(\bar{\xi} \subset F_{i+1}\) be the curve mapping to \(\xi\), so \(\bar{\xi}\) is an \(F_{i+1}\)-rigid curve.  We claim that \(\bar{\xi}\) is not \(\phi_{i+1}\)-periodic: indeed, if \(\phi_{i+1}^n(\bar{\xi}) = \bar{\xi}\) for some \(n\), then \(\bar{\xi}\) is contained \(F_{i+1} \cap \phi_{i+1}^n(F_{i+1})\).  But \(\pi_i(\bar{\xi}) = \xi\), contradicting (3) of Theorem~\ref{separateiterates}.

On the other hand, if \(\gamma \subset F_{i+1}\) is a \(\phi_{i+1}\)-periodic curve, with \(\phi_{i+1}^n(\gamma) = \gamma\), then \(\phi_i^n(\pi_i(\gamma)) = \pi_i(\phi_{i+1}^n(\gamma)) = \pi_i(\gamma)\), so \(\pi(\gamma)\) is \(\phi_i\)-periodic. Hence passing to the blow-up \(X_{i+1}\) does not introduce any new periodic curves. The curve \(\bar{\xi}\) is not \(\phi_{i+1}\)-periodic, so the number of periodic \(F_i\)-rigid curves decreases and \(N(\phi_{i+1}) < N(\phi_i)\).  By induction, we eventually reach a model \(Y = X_n\) for which there are no \(F_n\)-rigid \(\phi_n\)-periodic curves.  

Let \(\pi : Y \to X\) be the blow-down, and let \(G_i\) be the exceptional divisors of \(\pi\).  Then
\[
F \cdot g = (\pi^\ast E - \sum_i a_i G_i) \cdot g = E \cdot f - \sum_i a_i (G_i \cdot g).
\]
The right side is negative, because \(E \cdot f < 0\), \(a_i \geq 0\), and \(g\) is not contained in any of the \(G_i\).  
\end{proof}

\begin{lemma}
\label{ximovesalot}
Suppose that \(\phi : X \to X\) is an automorphism satisfying Condition (A$^\prime$).  Then there exists some nonzero \(n\) and a curve \(\xi \subset E_n \cap E\) such that \(\xi \subset E\) is an \(E\)-covering curve, and \(\xi \subset E_n\) is an \(E_n\)-covering curve.
\end{lemma}
\begin{proof}
Let \(\pi : Y \to X\) be the birational model constructed in Lemma~\ref{killperiodiccurves}, with \(F \subset Y\) the strict transform of \(E\).  Consider the set
\[
\Upsilon = \set{(\nu,\xi,n) : \psi^n(\nu) = \xi},
\]  
where \(\nu\) and \(\xi\) are irreducible curves in \(F\), and \(n\) is a nonzero integer.  As in Lemma~\ref{intersections}, the fact that \(F \cdot g < 0\) by Lemma~\ref{killperiodiccurves}(3) implies that \(F_n \cap F\) is nonempty for infinitely many \(n\), and so the set \(\Upsilon\) is infinite.

Suppose first that some curve \(\nu \subset F\) appears in infinitely many elements of \(\Upsilon\), so that there are infinitely many nonzero integers \(n_j\) with \(\psi^{n_j}(\nu) = \xi_j\) a curve in \(F\).  If there are distinct \(i\) and \(j\) for which \(\xi_i\) and \(\xi_j\) both coincide with some curve \(\xi\), then \(\psi^{n_i}(\nu) = \psi^{n_j}(\nu) = \xi\).  But then \(\psi^{n_i-n_j}(\nu) = \nu\).  Since there are no \(\psi\)-periodic \(F\)-rigid curves, the curve \(\nu\) must be \(F\)-covering.  But then \(\psi^{n_i-n_j}(\nu) = \nu\) is also an \(F_{n_i-n_j}\)-covering curve.

Otherwise, the curves \(\xi_j\) are all distinct.  There are only finitely many \(F\)-rigid curves, so there exist distinct \(i\) and \(j\) so that \(\xi_i\) and \(\xi_j\) are both \(F\)-covering curves. Then \(\psi^{n_i}(\nu) = \xi_i\) and \(\psi^{n_j}(\nu) = \xi_j\) implies that \(\psi^{n_i-n_j}(\xi_j) = \psi^{n_i}(\nu) = \xi_i\).  Then \(\xi_i\) is an \(F\)-covering curve and an \(F_{n_i-n_j}\)-covering curve.

Suppose instead that no curve \(\nu\) appears as the first entry of infinitely many elements of \(\Upsilon\), so that infinitely many different curves appear.  There are only finitely many \(F\)-rigid curves, so there exists an infinite sequence \(\nu_1,\nu_2,\ldots\) of \(F\)-covering curves such that \(\psi^{n_i}(\nu_i) = \xi_i\) is contained in \(F\).  If \(\xi_i\) is an \(F\)-covering curve for some value of \(i\), then \(\xi_i\) is both an \(F\)-covering curve and an \(F_{n_i}\)-covering curve.  If no \(\xi_i\) is \(F\)-covering, then there must exist distinct \(i\) and \(j\) with \(\xi_i = \xi_j\), as there are only finitely many \(F\)-rigid curves.  But then \(\psi^{n_i-n_j}(\nu_i) = \nu_j\), and \(\nu_j\) is both an \(F\)-covering curve and a \(F_{n_i-n_j}\)-covering curve.  

The map \(\pi\vert_{F_n} : F_n \to E_n\) is an isomorphism, and if \(\xi \subset F_n\) is an \(F_n\)-covering curve, then \(\pi(\xi) \subset E_n\) is an \(E_n\)-covering curve.   The above shows that there is a curve \(\xi \subset F\) that is a \(F\)-covering curve and an \(F_n\)-covering curve for some nonzero \(n\); the curve \(\pi(\xi)\) is then an \(E\)-covering curve and an \(E_n\)-covering curve, as required.
\end{proof}

\begin{remark}
The proof here is somewhat more convoluted than that sketched in the introduction.  The reason is that some care is required to handle the case when \(E\) is a ruled surface of type (R2b) with both \(E\)-rigid curves and \(E\)-covering curves, as in (6) of Example \ref{lotsofruledsurfaces}.  In essence, we first blow up any \(\phi\)-invariant \(E\)-rigid curves, and then argue as in the first case discussed in the introduction, when there do not exist any \(E\)-rigid curves.

It is worth considering what happens when \(E\) is of type (R2a).  The argument essentially hinges on property (2) of Lemma~\ref{killperiodiccurves}, the fact that after a sequence of blow-ups we can assume there are no \(\psi\)-periodic \(F\)-rigid curves.  In this case, the proof above is finished after the second paragraph, because there are no \(F\)-covering curves.  For a ruled surface of Type (R2a), the \(F_n\) would all be disjoint on the blow-up, which is impossible by (3) of the same lemma.  This case was illustrated in Figure~\ref{erigidcurve} of the introduction.
\end{remark}

We are at last in position to construct an invariant fibration for a map satisfying Condition (A$^\prime$).  

\begin{lemma}
\label{imprimitive}
Suppose that \(\phi\) is an automorphism satisfying Condition (A$^\prime$).  Then the map \(\phi : X \to X\) is imprimitive and satisfies Case 2(c) of Theorem~\ref{precisemain}.
\end{lemma}
\begin{proof}

Let \(\Hilb(X)\) be the Hilbert scheme of \(X\), with \(\tau : \Univ(X) \to \Hilb(X)\) the universal family.  Write \(\rho : \Univ(X) \subset X \times \Hilb(X) \to X\) for the evaluation map.  Given a closed subscheme \(V \subset X\), write \([V]\) for the corresponding point on the Hilbert scheme \(\Hilb(X)\).

If \(X\) is any variety and \(\phi : X \to X\) is an automorphism, there is an induced automorphism \(\phi_H : \Hilb(X) \to \Hilb(X)\), together with an induced automorphism of the universal family \(\phi_U : \Univ(X) \to \Univ(X)\).  The map \(\phi_H\) permutes the connected components of \(\Hilb(X)\).  Let \(\xi_n \subset E_n\) be an \(E_n\)-covering curve.  For every value of \(n\), the curve \(\xi_n\) moves in a flat family covering \(E_n\), and this deformation determines a curve \(\gamma_n \subset \Hilb(X)\).  

By the final part of Lemma~\ref{ximovesalot}, there is a curve \(\xi \subset E\) which is both an \(E\)-covering curve and an \(E_n\)-covering curve for some nonzero \(n\).  The curves \(\gamma_0\) and \(\gamma_n\) intersect at \([\xi]\) and so lie in the same connected component of \(\Hilb(X)\). Because \(\phi^{n}_H(\gamma_0) = \gamma_n\),  this component is invariant under \(\phi_H\).   Replacing \(\phi\) by \(\phi^{n}\), we may assume that the connected component containing \(\gamma_0\) is invariant under \(\phi_H\).

The connected component of the Hilbert scheme containing \([\gamma_0]\) has only finitely many irreducible components and these are permuted by the map \(\phi_H\), so we may replace \(\phi\) by a suitable iterate and assume that there is an irreducible component \(\Hilb_{[\xi]}(X)\) of \(\Hilb(X)\) containing all of the curves \(\gamma_n\) and fixed by \(\phi_H\).

Now, the curve \(\gamma_n = \phi_H^n(\gamma_0)\) is contained in \(\Hilb_{[\xi]}(X)\) for every \(n\).  Because the divisors \(E_n\) are distinct, so too are the curves \(\gamma_n\), and the irreducible component \(\Hilb_{[\xi]}(X)\) contains infinitely many curves. It follows that \(\Hilb_{[\xi]}(X)\) has dimension at least \(2\), and so 
\[ 
\dim H^0(\xi,N_{\xi/X}) = \dim T_{[\xi]} \Hilb_{[\xi]}(X) \geq \dim \Hilb_{[\xi]}(X) \geq 2.
\]
Next we show that in fact equality holds in the above, so that \(\Hilb_{[\xi]}(X)\) has dimension exactly \(2\).  There is a short exact sequence of normal bundles
\[
\xymatrix{
0 \ar[r] & N_{\xi/E} \ar[r] & N_{\xi/X} \ar[r] & N_{E/X}\vert_\xi \ar[r] & 0 }
\]
The first term \(N_{\xi/E}\) is a trivial \(\cO_\xi\).  Let \(\xi_n\) be a general \(E_n\)-covering curve.  The intersection \(E \cap E_n \subset E_n\) is a union of \(E_n\)-covering curves.  These are disjoint from \(\xi_n\), and so \(E \cdot \xi_n = 0\).  Since \(\xi\) and \(\xi_n\) are numerically equivalent, we have \(E \cdot \xi = 0\) as well, so that \(N_{E/X}\vert_\xi\) has degree \(0\).   Now consider the exact sequence in cohomology
\[
\xymatrix{
0 \ar[r] & H^0(\xi,N_{\xi/E}) \ar[r] & H^0(\xi,N_{\xi/X}) \ar[r] & H^0(\xi,N_{E/X}\vert_\xi)  \ar[r]^\delta & H^1(\xi,N_{\xi/E}) }
\]
The first term has dimension \(1\), while the third term \(H^0(\xi,N_{E/X}\vert_\xi)\) has dimension \(1\) if \(N_{E/X}\vert_\xi\) is trivial and \(0\) otherwise.  This yields \(H^0(\xi,N_{\xi/X}) \leq 2\), with equality if and only if \(\dim N_{E/X}\vert_\xi = 1\) and the map \(\delta\) is \(0\). We have already seen \(H^0(\xi,N_{\xi/X}) \geq 2\), and so it must be that equality holds and \(N_{E/X}\vert_\xi\) is trivial.  The boundary map \(\delta : H^0(\xi,N_{E/X}\vert_\xi) \to H^1(\xi,N_{\xi/E})\) then computes the extension class of the normal bundle sequence, and since \(\delta\) is zero the extension of normal bundles is split and \(N_{\xi/X} \cong \cO_\xi \oplus \cO_\xi\) is a trivial rank-\(2\) bundle.
We conclude that \(\dim \Hilb_{[\xi]}(X) = 2\), that \([\xi]\) is a smooth point, and that \(N_{\xi/X} \cong \cO_\xi \oplus \cO_\xi\) is trivial.  In particular, \(\Hilb_{[\xi]}(X)\) is generically smooth.

Take \(\Univ_{[\xi]}(X)\) to be the component of \(\Univ(X)\) lying over \(\Hilb_{[\xi]}(X)\).  The image of \(\rho_\xi : \Univ_{[\xi]}(X) \to X\) contains every \(E_n\), and these divisors are Zariski dense, so the map \(\rho_\xi\) is surjective.  Because \(X\) and \(\Univ_{[\xi]}(X)\) both have dimension \(3\), the map \(\rho_\xi\) is generically finite.

We claim next that \(\rho_\xi\) is in fact birational. Suppose that \(\rho_\xi\) is generically \(d\) to \(1\), with \(d > 1\). Because \(\xi\) is smooth and irreducible, the irreducible component \(\Hilb_{[\xi]}(X)\) is birational to an irreducible component \(\Chow_{[\xi]}(X)\) of the corresponding Chow variety of \(X\), parametrizing cycles equivalent to \(\xi\)~\cite[Cor.\ I.6.6.1]{kollarrationalcurves}.  The map \(\Chow_{[\xi]}(X) \to X\) is generically \(d\) to \(1\) as well. (More simply, there is an open set \(U \subset \Hilb_{[\xi]}(X)\) parametrizing smooth cycles numerically equivalent to \(\xi\), and the preimage of a general point of \(X\) under \(\rho_\xi^{-1}\) is given by \(d\) points in \(\tau^{-1}(U)\).)

The divisors \(E_n\) are dense on \(X\), so there exists a point \(x\) on some \(E_n\) for which the preimage of \(\rho_\xi\) consists of \(d\) distinct points.  One of points of \(\rho_\xi^{-1}(x)\) parametrizes the \(E_n\)-covering curve \(\xi_n\) through \(x\).  Suppose that one of the others parametrizes a cycle \(\eta\) on \(X\).  We have \(\eta \cdot E_n = 0\) because \(\eta\) is numerically equivalent to \(\xi\) on \(X\).  Since \(\eta\) passes through the point \(x \in E_n\), it must be that \(\eta\) is contained in \(E_n\).  However, \(E_n \cdot f_n = -1\) and \(E_n \cdot \xi = 0\).  As \(f\) and \(\xi\) generate the two rays on \(\NEb(E_n)\), the intersection \(E_n \cdot \eta\) must be negative unless \(\eta\) is numerically equivalent to \(\xi\) on \(E_n\). But then \((\xi \cdot \eta)_{E_n} = (\xi \cdot \xi)_{E_n} = 0\).  This is impossible, because \(\xi\) and \(\eta\) meet at the point \(x \in E_n\). Consequently we must have \(d=1\), so that  \(\rho : \Univ_{[\xi]}(X) \to X\) is birational, and there exists an inverse map \(\rho_{\xi}^{-1} : X \rat \Univ_{[\xi]}(X)\).

The automorphism \(\phi_U : \Univ_{[\xi]}(X) \to \Univ_{[\xi]}(X)\) permutes the fibers of \(\Univ_{[\xi]}(X) \to \Hilb_{[\xi]}(X)\).  The schemes \(\Hilb_{[\xi]}(X)\) and \(\Univ_{[\xi]}(X)\) might not be varieties, for they could be nonreduced away from \([\xi]\).  However, taking the induced maps on the underlying reduced schemes, we obtain a map \(\pi : X \rat \Univ_{[\xi]}(X)_{\red} \to \Hilb_{[\xi]}(X)_{\red}\) which realizes \(\phi : X \to X\) as an imprimitive automorphism over a \(2\)-dimensional projective variety.
\[
\xymatrix@C=0em{
& X \ar[rr]^\phi \ar@{-->}[dd] && X \ar@{-->}[dd] \\ 
\Univ_{[\xi]}(X)_{\red}  \ar[rr] \ar[ur]^{\rho_\xi} \ar[dr]_\tau && \Univ_{[\xi]}(X)_{\red} \ar[ur]^{\rho_\xi} \ar[dr]_\tau \\ 
& \Hilb_{[\xi]}(X)_{\red} \ar[rr]^{\phi_H} &&  \Hilb_{[\xi]}(X)_{\red}
} 
\]

Take \(X^\prime = \Univ_{[\xi]}(X)_{\red}\) and \(S = \Hilb_{[\xi]}(X)_{\red}\).  The map \(\rho : X^\prime \to S\) has all fibers one-dimensional, since \(\Univ_{[\xi]}(X) \to \Hilb_{[\xi]}(X)\) is flat.  This shows that \(\phi\) satisfies the conclusions of Case 2(c) of Theorem~\ref{precisemain}.  
\end{proof}

\begin{remarks}
If one is willing to allow \(X^\prime\) in Theorem~\ref{precisemain}, Case 2(c) to be a non-reduced scheme, then the map \(\rho : X^\prime \to S\) may be assumed to be flat.  However, it is possible that flatness is lost after replacing \(\Hilb_{[\xi]}(X)\) with \(\Hilb_{[\xi]}(X)_{\red}\).

The map \(\phi_H : S \to S\) is a positive entropy automorphism of a surface, so a general point has dense orbit.  The fiber over the point \(\phi_H^n([\xi]) \in S\) is a smooth curve isomorphic to \(\xi\), and so there is a Zariski dense set of fibers of \(X^\prime \to S\) which are isomorphic to \(\xi\). It is possible, however, that there are some singular or nonreduced fibers.

The fact that \(\Hilb_{[\xi]}(X)\) generically parametrizes smooth curves with trivial normal bundle does not itself imply that \(\Univ_{[\xi]}(X) \to X\) must be birational; it really is necessary to use the specific geometry of this setting. If \(Y \subset \P^4\) is a general smooth cubic threefold, then through a general point there are six lines \(\ell\), each with trivial normal bundle.  The component \(\Hilb_{[\ell]}(Y)\) in this case is a smooth surface of general type (a so-called Fano surface), and the universal family \(\Univ_{[\ell]}(Y) \to Y\) is generically \(6\) to \(1\).
\end{remarks}

\begin{example}
It is worth pointing out an example where the map \(\Univ_{[\xi]}(X)_\red \to X\) is not an isomorphism.  Consider again Example~\ref{productex}, with \(\sigma : S \to S\) an automorphism of a rational surface and \(\sigma \times \id : S \times C \to S \times C\) an automorphism.  Let \(p \in S\) be a fixed point of \(\sigma\) not contained in any \((-1)\)-curve, and let \(q\) be any point on \(C\). Take \(X\) to be the blow-up of \(S \times C\) at \((p,q)\), with exceptional divisor \(F\), so that \(\sigma \times \id\) lifts to an automorphism \(\phi : X \to X\).
The divisorial contraction \(\pi : X \to Y\) may be taken to blow down \(\ell \times C\), where \(\ell \subset S\) is a \((-1)\)-curve.  The exceptional divisor \(E\) of \(\pi\) is disjoint from the exceptional divisor \(F\) of \(X \to S \times C\).

As the curve \(\xi = p^\prime \times C\) moves to \(p \times C\), the flat limit is given as the union of the strict transform of \(p \times C\) and a line in the exceptional divisor \(F\), which depends on the direction from which \(p^\prime\) approaches \(p\).   The corresponding component \(\Hilb_{[\xi]}(X)\) is isomorphic to \(\Bl_p S\), and the universal family \(\Univ_{[\xi]}(X)_\red \to X\) is birational.  However, the preimage in \(\Univ_{[\xi]}(X)_\red\) of \(p \times z\) for any \(z \neq q\) is \(1\)-dimensional.
\end{example}

Lemma~\ref{imprimitive} completes the proofs of the theorems claimed in the introduction.

\thmblowupautos*
\begin{proof}
The proof is by induction on \(\rho(X/M)\), the number of blow-ups used in constructing \(X\).  When \(\rho(X/M) = 0\), we have \(X = M\) and there is nothing to check.  Otherwise, let \(\pi : X \to Y\) be the last of the sequence of blow-ups in the construction of \(X\), with exceptional divisor \(E\).  If \(E\) has infinite orbit under \(\phi\), then \(\pi\) must be the blow-up of a curve by Lemma~\ref{notpoint}, and \(\phi\) must be imprimitive by Lemma~\ref{imprimitive}.   Otherwise, some iterate of \(\phi\) descends to an automorphism  \(\psi : Y \to Y\).  Since \(Y\) is also a smooth blow-up of \(M\) and has smaller Picard rank, we conclude by induction that \(\psi\) is imprimitive, which means that \(\phi\) is imprimitive as well.
\end{proof}

\thmmain*
\begin{proof}[Proof of Theorems~\ref{main} and \ref{precisemain}]
If \(\phi\) does not satisfy Condition (A), the theorem was proved in Section~\ref{geometric} as a consequence of Lemmas~\ref{nefcanonical}, \ref{morifiberspace}, and \ref{periodicexceptional}.  If \(\phi\) does satisfy Condition (A), it satisfies Condition (A$^\prime$) by Lemmas~\ref{intersections},~\ref{notpoint}, and~\ref{numdim}.  But if \(\phi\) satisfies Condition (A$^\prime$), it satisfies 2(c) of Theorem~\ref{precisemain} by Lemma~\ref{imprimitive}.
\end{proof}

\thminduction*
\begin{proof}[Proof of Corollary~\ref{induction}]
Let \(\pi : X \to Y\) be a contraction in the MMP for \(X\).  Since \(\phi\) is primitive, \(\pi\) can not be a Mori fiber space by Lemma~\ref{morifiberspace}.
So \(\pi\) must be a divisorial contraction.  If the exceptional divisor \(E\) had infinite orbit, then \(\phi\) would be imprimitive.  Hence \(E\) is \(\phi\)-periodic, and some iterate of \(\phi\) descends to an imprimitive automorphism of \(Y\).  If \(Y\) is not smooth, then the claim is proved.  Otherwise, we replace \(X\) with \(Y\) and repeat the argument; since the Picard rank decreases at every step, the process must eventually yield a non-smooth threefold on which \(\phi\) induces an automorphism.
\end{proof}

\thminvariantstuff*
\begin{proof}[Proof of Corollary~\ref{invariantstuff}]
If \(K_X\) is not numerically trivial and \(X\) admits a primitive, positive entropy automorphism, then \(K_X\) is not nef by Lemma~\ref{nefcanonical}.  Let \(\pi : X \to Y\) be a contraction of the \(K_X\)-MMP.  The map \(\pi\) is not a Mori fiber space, because \(X\) admits a primitive automorphism of infinite order.  Hence \(\pi\) is a divisorial contraction.  If the exceptional divisor \(E\) of \(\pi\) is not \(\phi\)-periodic, then \(\pi(E)\) is a curve by Lemma~\ref{notpoint}, and \(\phi\) is imprimitive by Lemma~\ref{imprimitive}.  Hence \(E\) must be \(\phi\)-periodic, and the divisor \(\bigcup_n \phi^n(E)\) is \(\phi\)-invariant.
\end{proof}
Note that in Corollary~\ref{invariantstuff} it is not necessary to replace \(\phi\) by an iterate to obtain the conclusion.  If \(E\) is invariant for some iterate \(\phi^m\), then \(\bigcup_{n=0}^{m-1} \phi^n(E)\) is invariant for \(\phi\).

\section{The problem with flips}
\label{singularcase}

\newcommand{\Xb}{X_{\mathrm{sing}}}
\newcommand{\Xt}{X_{\mathrm{term}}}
\newcommand{\Xs}{X_{\mathrm{smth}}}

\newcommand{\Sb}{S_{\mathrm{sing}}}
\newcommand{\Sc}{X_{\mathrm{can}}}
\newcommand{\Ss}{X_{\mathrm{smth}}}

A shortcoming of the proof of Theorem~\ref{precisemain} is that if a divisorial contraction \(\pi : X \to Y\) gives rise to a singular variety \(Y\), no further progress is possible. There are two basic obstructions to extending the arguments to the singular case. First, running the MMP on \(Y\) might require performing a flip \(\sigma : Y \rat Y^+\).  If the flipping curve \(C \subset Y\) has infinite orbit under \(\phi\), then \(\phi\) induces only a pseudoautomorphism of \(Y^+\).   Second, even if there is a divisorial contraction \(\pi : Y \to Z\), the exceptional divisor might not be isomorphic to a smooth ruled surface \(E\) and Lemma~\ref{rationality} does not apply; some contractions of this type are described in \cite{kollarmori}.

To illustrate the difficulty with flips, we describe the first steps of a run of the the MMP for Example~\ref{otrational} of Oguiso and Truong.  Let \(\omega = (-1+\sqrt{3} i)/2\) and \(E = \C/(\Z \oplus \omega \Z)\).    Consider the action of \(\tau : E \to E\) given by multiplication by \(-\omega\), a sixth root of unity.  There are six points on \(E\) with nontrivial stabilizer under the action of \(\tau\):
\begin{enumerate}
\item \(\Stab(x) = \ang{\tau}\): \(\set{0}\),
\item \(\Stab(x) = \ang{\tau^2}\): \(\set{(2+\omega)/3,(1+2\omega)/3}\),
\item \(\Stab(x) = \ang{\tau^3}\): \(\set{1/2,(1+\omega)/2,\omega/2}\).
\end{enumerate}

Consider the threefold \(E \times E \times E\), with the diagonal action of \(\tau\), denoted \(\tau_\Delta : E \times E \times E \to E \times E \times E\).  Let \(r : E \times E \times E \to \Xb\) be the quotient by this cyclic action.  A point \((x,y,z)\) on \(E \times E \times E\) is fixed by \(\tau^k_\Delta\) if and only if each of its entries is fixed by \(\tau^k\), so the points on \(E \times E \times E\) with nontrivial stabilizer are:
\begin{enumerate}
\item \(\Stab(x,y,z) = \ang{\tau_\Delta}\).  There is a unique point of this form, giving rise to a singularity of type \(\nicefrac{1}{6}(1,1,1)\) on \(\Xb\).
\item \(\Stab(x,y,z) = \ang{\tau_\Delta^2}\).  There are \(3\) points fixed by \(\tau^2\) on \(E\), and hence \(3^3 -1 = 26\) points with stabilizer \(\ang{\tau^2_\Delta}\).  The orbits of these points have size \(2\), giving \(13\) singularities of type \(\nicefrac{1}{3}(1,1,1)\) on \(\Xb\).
\item \(\Stab(x,y,z) = \ang{\tau_\Delta^3}\).  There are \(4^3-1 = 63\) points with stabilizer \(\tau^3_\Delta\).  The orbits have size \(3\), giving \(21\) singularities of type \(\nicefrac{1}{2}(1,1,1)\) on \(\Xb\).
\end{enumerate}

Let us briefly recall some standard facts about singularities of type \(\nicefrac{1}{d}(1,1,1)\).  Let \(\omega\) be a \(d\)th root of unity,
and let \(\Z/d\Z\) act on \(\C[x,y,z]\) by multiplication by \(\omega\) in each variable.  The ring of invariants of the action is generated by monomials \(x^i y^j z^k\) with \(i+j+k =d\), and so the singularity is isomorphic to that of the projective cone over the degree-\(d\) Veronese embedding \(\P^2 \to \P^N\).  

\newcommand{\Yr}{Y_{\text{res}}}
\newcommand{\Yc}{Y_{\text{cone}}}

Write \(\Yc\) for this cone, and let \(\pi : \Yr = \P_{\P^2}(\cO \oplus \cO(d)) \to \Yc\) be blow-up at the cone point. The singularity \(\nicefrac{1}{d}(1,1,1)\) is resolved by a single blow-up, and the exceptional divisor \(E\) is isomorphic to \(\P^2\), with normal bundle \(\cO_{\P^2}(-d)\).  Write \(K_{\Yr} = \pi^\ast K_{\Yc} + a E\), so \(K_{\Yr}+E = \pi^\ast K_{\Yc} + (a+1)E\).  By adjunction we have \(K_E = (a+1)E\vert_E\).  But \(K_E = \cO_{\P^2}(-3)\), while \(E\vert_E =  \cO_{\P^2}(-d)\), yielding \(a = \frac{3}{d}-1\).

When \(d=2\) we have \(a=\frac{1}{2}\), which shows that the singularities of type \(\nicefrac{1}{2}(1,1,1)\) are terminal.  When \(d=3\) we obtain \(a = 0\), and so \(\nicefrac{1}{3}(1,1,1)\) is canonical but not terminal.  At last, when \(d=6\), this yields \(a=-\frac{1}{2}\), and so the singular point of type \(\nicefrac{1}{6}(1,1,1)\) is klt but not canonical. The map \(\pi : \Xs \to \Xb\) which blows up each singular point is a resolution.   Write \(E_6\) for the  exceptional divisor over the \(\nicefrac{1}{6}(1,1,1)\) point, and \(E_3^i\) and \(E_2^j\) for the exceptional divisors over the singular points of type \(\nicefrac{1}{3}(1,1,1)\) and \(\nicefrac{1}{2}(1,1,1)\). Let \(\ell_6\), \(\ell_3^i\), and \(\ell_2^j\) be lines in the corresponding exceptional divisors. Since \(r\) is \'etale in codimension \(1\) and \(K_{E \times E \times E} = 0\),  the computation of the discrepancies gives \(K_{\Xs} = \frac{1}{2} \sum_i E^i_2 -\frac{1}{2} E_6\). 

Now consider a run of the MMP on \(\Xs\).  Each of the curves \(\ell_2^j\) has \(K_{\Xs} \cdot \ell_2^j = -1\) and spans an extremal ray on \(\NEb(\Xs)\).  
There is a sequence of divisorial contractions of type (E5), contracting all of the divisors \(E_2^j\) and yielding a variety \(\Xt\).  The model \(\Xt\) can be obtained directly from \(\Xb\) by resolving the singularities of types \(\nicefrac{1}{6}(1,1,1)\) and \(\nicefrac{1}{3}(1,1,1)\), but not blowing up the terminal singularities of type \(\nicefrac{1}{2}(1,1,1)\). The canonical class is given by \(K_{\Xt} = -\frac{1}{2} E_6\), and the anticanonical class is effective.

Let \(C\) be the strict transform on \(\Xt\) of \(\bar{C} = r(E \times 0 \times 0) \subset \Xb\).  Since \(\bar{C}\) passes through the singularity of type \(\nicefrac{1}{6}(1,1,1)\), \(C\) meets the exceptional divisor \(E_6\), and so \(K_{\Xt} \cdot C < 0\).  We claim that in fact \(C\) spans a \(K_{\Xt}\)-negative extremal ray on \(\NEb(\Xt)\).

Let \(\Sb = (E \times E)/\tau\), and let \(\pi_{23} : \Xb \to \Sb\) be the projection onto the last two coordinates.  Consider the composition \(\bar{\pi}_{23} : \Xt \to \Xb \to \Sb\).  The fiber of \(\pi_{23}\) over \((0,0)\) is the curve \(\bar{C}\).  There are three singular points of \(\Xb\) on \(r(E \times 0 \times 0)\), of types \(\nicefrac{1}{6}(1,1,1)\) and \(\nicefrac{1}{3}(1,1,1)\), and \(\nicefrac{1}{2}(1,1,1)\).  The first two of these are blown up on \(\Xt\), and so \(\bar{\pi}_{23}^{-1}(0,0) \subset \Xt\) is the union of \(C\) and two exceptional divisors \(E_6\) and \(E_3^0\), which are disjoint and meet \(C\) at one point each.  Since the relative canonical class \(K_{\Xt/\Sb}\) is \(\bar{\pi}_{23}\)-numerically equivalent to \(-\frac{1}{2}E_6\), the only \(K_{\Xt/\Sb}\)-negative curve contracted by \(\bar{\pi}_{23}\) is \(C\).  In particular, there exists a flip \(\sigma : \Xt \rat X^+\) of \(C\) over \(\Sb\).  The same map is a flip for the \(K_{\Xt}\)-MMP.  Observe that \(C\) passes through a singular point of \(\Xt\), as any flipping curve on a terminal threefold must.

It is straightforward to explicitly describe the flip \(\sigma\) by a resolution; the map is locally a familiar one, described e.g.\ in~\cite[\S 6.20]{debarre}.  Let \(\Xt^0\) denote the threefold obtained by blowing up the unique singular point on \(C\), with exceptional divisor \(E_2^0\). There is a resolution of \(\sigma\) illustrated in the following diagram.
\[
\xymatrix{
& W \ar[dl]_h \ar[dr]^i \\
\Xt^0 \ar@{-->}[rr]^\psi \ar[d]^\pi && X^\prime \ar[d]_{\pi^+} \\
\Xt \ar@{-->}[rr]^\sigma \ar[dr]_f && X^+ \ar[dl]^g \\
& Z
}
\]
The map \(\pi\) is the blow-up at the singular point on \(C\), with exceptional divisor \(E_2^0\) isomorphic to \(\P^2\).  The strict transform \(C^0\) of \(C\) on \(\Xt^0\) is a rational curve which does not meet any of the singular points of \(\Xt^0\).  The normal bundle of \(C^0\) is \(\cO_{\P^1}(-1) \oplus \cO_{\P^1}(-1)\), and the map \(\psi : \Xt^0 \rat X^\prime \) is the standard flop of \(C^0\): \(h\) blows up \(C^0\), with exceptional divisor \(F\) isomorphic to \(\P^1 \times \P^1\), and \(i\) contracts \(F\) along the other ruling.   The strict transform of \(E_2^0\) on \(X^\prime\) is isomorphic to a Hirzebruch surface \(\F_1\), and \(\pi^+\) is the contraction of \(E_2^0\) to \(\P^1\).

There is an action of \(\SL_3(\Z)\) on \(\Xt\) by automorphisms, and the image of \(C\) under any automorphism is another flipping curve.  Since \(C\) has infinite orbit under the action of this group, there are infinitely many flipping curves on \(\Xt\).  If \(\phi\) is such an automorphism, the induced map \(\phi^+ : X^+ \rat X^+\) might no longer be an automorphism; it becomes indeterminate along the flipped curve.   We next turn our attention to Question~\ref{kmmquestion} from the introduction.

Observe that for a surface of non-negative Kodaira dimension, the number of \(K_X\)-negative extremal rays on \(\NEb(X)\) is always finite: if \(K_X\) is numerically equivalent to an effective divisor \(D\), any \(K_X\)-negative irreducible curve must be one of the finitely many components of \(D\).  If \(\dim X = 3\) and \(\kappa(X) \geq 0\), there are again only finitely divisors that can be contracted, and a given divisor can be contracted in only finitely many ways.  Since a non-uniruled threefold can not admit a Mori fiber space structure, a variety with infinitely many \(K_X\)-negative rays must contain infinitely many flipping curves.  The example \(Y\) will be constructed as a branched cover of the variety \(\Xt\).

\thmmanyflips
\begin{proof}
Let \(\bar{\pi}_3 : \Xt \to \Xb \to E/\tau \cong \P^1\) be the third projection, with \(0 \in \P^1\) the image of \(0 \in E\).  The curve \(C\) lies in the fiber \(\bar{\pi}_3^{-1}(0)\).  There are infinitely many flipping curves for the \(\Xt\)-MMP over \(\P^1\), since the orbit of \(C\) under the subgroup of \(\Aut(\Xt)\) induced by matrices of the form
\[
M = \left( \begin{array}{c|c} \SL_2(\Z) & 0 \\\hline 0 & 1 \end{array} \right)
\]
is still infinite, and this subgroup commutes with \(\bar{\pi}_3\).  All the curves in this orbit are \(K_{\Xt}\)-flipping curves contained in the fiber \(\bar{\pi}_3^{-1}(0)\).  Now, let \(\Gamma\) be a curve of genus at least \(1\) with a map \(\beta : \Gamma \to \P^1\) not ramified over any point of the finite set \(\pi_3(\Sing \Xb)\). Let \(\bar{\beta} : Y \to \Xt\) be the branched cover of \(\Xt\) constructed as the pull-back family \(Y = \Xt \times_{\P^1} \Gamma\).
\[
\xymatrix{
Y \ar[r]^{\bar{\beta}} \ar[d]^{\bar{\pi}} & \Xt \ar[d]_{\bar{\pi}_3} \\
\Gamma \ar[r]^\beta & \P^1
}
\]
Because the ramification locus of \(\bar{\beta}\) is disjoint from the singularities of \(X\), the variety \(Y\) has only terminal singularities. 
The general fibers of \(\bar{\pi}\) are smooth abelian surfaces \(E \times E\), and since \(\Gamma\) is not rational, through a general point of \(Y\) there does not pass any rational curve, so \(Y\) is not uniruled.   Let \(E_6^1,\ldots,E_6^d\) be the preimages of the divisor \(E_6\) on \(\Xt\), where \(d = \deg(\beta)\).  We have
\[
K_Y = \bar{\beta}^\ast K_{\Xt} + R = -\frac{1}{2} \sum_{i=1}^d E_6^i + \bar{\pi}^\ast R_\Gamma,
\]
where \(R_\Gamma \subset \P^1\) is the ramification divisor of \(\beta\).
Let \(\gamma\) be any point with \(\beta(\gamma) = 0\).  The fiber of \(\bar{\pi}\) over \(\gamma\) is isomorphic to the fiber of \(\Xt\) over \(0 \in \P^1\), and the restriction of \(K_{Y/\Gamma}\) to this fiber is isomorphic to the restriction of \(K_{\Xt/\P^1}\).  The curves in \(Y\) which map to \(C\) and its orbit under \(\Aut(\Xt)\) are all contracted by \(\bar{\pi}\), and so give \(K_Y\)-negative curves which are extremal on \(\NEb(Y/\Gamma)\).  These curves can be flipped over \(\Gamma\), and indeed define \(K_Y\)-flipping contractions.  As a result, there are infinitely many \(K_Y\)-negative extremal rays on \(\NEb(Y)\).
\end{proof}

The fiber \(\bar{\pi}^{-1}(\gamma)\) is a union of six two-dimensional components, illustrated in Figure~\ref{flipspicture}.  One is a rational surface \(S_0\), which is a partial desingularization of the quotient \((E \times E)/\tau\). There are five singularities of \(Y\) of type \(\nicefrac{1}{2}(1,1,1)\) lying on \(S_0\); as singularities of the surface, these points are ordinary double points.  The other five components are the preimages on \(Y\) of the exceptional divisors of the map \(\Xt \to \Xb\), and are mutually disjoint. One of these, \(E_6\), is the resolution of a \(\nicefrac{1}{6}(1,1,1)\) singularity, while the other four, \(E_3^i\), are resolutions of \(\nicefrac{1}{3}(1,1,1)\) singularities.  The divisor \(E_6\) intersects \(S_0\) along a \((-6)\)-curve in \(S_0\), while the \(E_3^i\) intersect along \((-3)\)-curves.  The flipping curves are contractible curves on \(S_0\) which pass through a singular point \(Y\) and meet the divisor \(E_6\).  These lift to certain \((-1)\)-curves on the minimal resolution of \(S_0\).

\begin{figure}[htb]
\centering
\begin{tikzpicture}

\def\baseangle{15}
\draw (-3,0) to [out=\baseangle,in=180-\baseangle] (0,0) to [out=-\baseangle,in=180+\baseangle] (3,0);
\node [below] at (0,0) {$\gamma$};
\node [above] at (-1,0.2) {$\Gamma$};
\fill (0,0) circle [radius=2pt];

\def\ellipsehole{65}
\draw (-2,3) ellipse (0.5 and 1.2);
\draw (-2,2.4) to [out=\ellipsehole,in=-\ellipsehole] (-2,3.6);
\draw [shorten >=-3,shorten <=-3](-2,2.4) to [out=180-\ellipsehole,in=180+\ellipsehole] (-2,3.6);
\node at (-2,4.5) {$E \times E$};

\def\ellipsehole{65}
\draw (2,3) ellipse (0.5 and 1.2);
\draw (2,2.4) to [out=\ellipsehole,in=-\ellipsehole] (2,3.6);
\draw [shorten >=-3,shorten <=-3](2,2.4) to [out=180-\ellipsehole,in=180+\ellipsehole] (2,3.6);

\draw (0,3) ellipse (0.7 and 1.2);
\node at (0.4,1.65) {$S_0$};

\def\sixgray{80}
\def\xcenter{0}
\def\ycenter{3.8}
\def\curvelen{0.5}
\def\height{0.7}
\def\angle{15}
\draw [fill=gray!\sixgray] (\xcenter-\curvelen/2,\ycenter) -- (\xcenter+\curvelen/2,\ycenter) to [out=90-\angle,in=0] (\xcenter,\ycenter+\height) to [out=180,in=90+\angle] (\xcenter-\curvelen/2,\ycenter);
\node at (0.5,4.5) {$E_6$};

\node at (0,2.25) {$\times$};

\def\threegray{20}

\def\curvelen{0.5}
\def\height{0.7}
\def\angle{15}
\def\xcenter{-0.5}
\def\ycenter{3.5}
\draw [fill=gray!\threegray] (\xcenter,\ycenter-\curvelen/2) -- (\xcenter,\ycenter+\curvelen/2) to [out=180-\angle,in=90] (\xcenter-\height,\ycenter) to [out=270,in=180+\angle](\xcenter,\ycenter-\curvelen/2);

\def\xcenter{-0.5}
\def\ycenter{2.5}
\def\curvelen{0.5}
\def\height{0.7}
\def\angle{15}
\draw [fill=gray!\threegray] (\xcenter,\ycenter-\curvelen/2) -- (\xcenter,\ycenter+\curvelen/2) to [out=180-\angle,in=90] (\xcenter-\height,\ycenter) to [out=270,in=180+\angle](\xcenter,\ycenter-\curvelen/2);

\def\xcenter{0.5}
\def\ycenter{2.5}
\def\curvelen{0.5}
\def\height{-0.7}
\def\angle{15}
\draw [fill=gray!\threegray] (\xcenter,\ycenter-\curvelen/2) -- (\xcenter,\ycenter+\curvelen/2) to [out=\angle,in=90] (\xcenter-\height,\ycenter) to [out=-90,in=-\angle](\xcenter,\ycenter-\curvelen/2);

\def\xcenter{0}
\def\ycenter{2.9}
\def\curvelen{0.5}
\def\height{0.7}
\def\angle{15}
\draw [fill=gray!\threegray ] (\xcenter-\curvelen/2,\ycenter) -- (\xcenter+\curvelen/2,\ycenter) to [out=90-\angle,in=0] (\xcenter,\ycenter+\height) to [out=180,in=90+\angle] (\xcenter-\curvelen/2,\ycenter);

\node at (1,3) {$E_3^i$};

\node at (0,2.25) {$\times$};
\node at (0.25,2.65) {$\times$};
\node at (-0.34,3.65) {$\times$};
\node at (-0.5,3.0) {$\times$};
\node at (0.42,3.5) {$\times$};

\draw [very thick] (0,2) -- (0,4);
\node at (-0.24,2.58) {$C$};

\draw [densely dotted] (0,0.15) -- (0,1.7);
\draw [densely dotted] (-2,0.3) -- (-2,1.7);
\draw [densely dotted] (2,-0.1) -- (2,1.7);

\fill (0,2.9) circle [radius=2pt];
\fill (0,3.8) circle [radius=2pt];
\end{tikzpicture}
\caption{The family \(\bar{\pi} : Y \to \Gamma\)}
\label{flipspicture}
\end{figure}

\section{Threefolds with two commuting automorphisms}

If \(X\) is a projective threefold, the rank of an abelian subgroup of \(\Aut(X)\) is at most \(2\)~\cite{dinhsibony}.  The study of threefolds achieving this upper bound, i.e.\ admitting two commuting, positive entropy automorphisms, is a problem of particular interest.  When \(X\) is not rationally connected, it is a result of Zhang that \(X\) must be birational to a torus quotient~\cite{dqzhangmaxrank}.  

In this section we point out applications of Theorem~\ref{main} in the case that \(X\) is smooth and rationally connected.  These rely on the following result in the two-dimensional case.

\begin{proposition}[\cite{dinhsibony}]
\label{commuteonsurface}
Suppose that \(S\) is a smooth projective surface and that \(\phi\) and \(\psi\) are two commuting, positive entropy automorphisms of \(S\).  Then there exist integers \(m\) and \(n\) so that \(\phi^m = \psi^n\).
\end{proposition}

In analogy with Theorem~\ref{main}, we show if \(\phi : X \to X\) and \(\psi : X \to X\) are commuting, positive entropy automorphisms of a smooth threefold, then either \(\phi\) and \(\psi\) must both be imprimitive over the same surface (in which case Proposition~\ref{commuteonsurface} gives further results), or there is a singular variety \(Y\) on which \(\phi\) and \(\psi\) both induce automorphisms.  In this case we can say nothing more, though it is perhaps evidence that even in the rationally connected case, quotient constructions may be the best source of examples.

\begin{theorem}
Suppose that \(X\) is a smooth, rationally connected threefold and that \(\phi\) and \(\psi\) are commuting, positive entropy automorphisms of \(X\).  After replacing both \(\phi\) and \(\psi\) with appropriate iterates, either:
\begin{enumerate}
\item there exists a singular threefold \(Y\) with terminal singularities and \(\rho(Y) < \rho(X)\) such that \(\phi\) and \(\psi\) both induce automorphisms of \(Y\); or
\item there exists a map \(\pi : X \rat V\) with \(\dim V < \dim X\) and an automorphism \(\tau \in \Aut(X/V)\) such that \(\phi = \psi \circ \tau\).
\end{enumerate}
If \(M\) is a smooth threefold with no positive entropy automorphisms and \(X\) is a smooth blow-up of \(M\), (2) must hold.
\end{theorem}
\begin{proof}
Let \(\pi : X \to Y\) be the first step of the MMP applied to \(X\).  Suppose first that \(\pi : X \to Y\) is a Mori fiber space.  By Lemma~\ref{morifiberspace}, after replacing \(\phi\) and \(\psi\) by suitable iterates, we may assume that \(Y\) is a surface and that both \(\phi\) and \(\psi\) descend to positive entropy automorphisms \(\bar{\phi}\) and \(\bar{\psi}\) on \(Y\).  Again replacing \(\phi\) and \(\psi\) by iterates, by Proposition~\ref{commuteonsurface} we may assume that \(\bar{\phi} = \bar{\psi}\).  Then \(\tau = \phi \circ \psi^{-1}\) is an automorphism of \(X\) over \(Y\), and outcome (2) of the theorem is satisfied.

We must now treat the case in which \(\pi : X \to Y\) is a divisorial contraction, with exceptional divisor \(E\).  Suppose that \(E\) is either \(\phi\)-periodic or \(\psi\)-periodic; replacing by an iterate and exchanging \(\phi\) and \(\psi\) if needed, we may without loss of generality assume that \(E\) fixed by \(\phi\).  Then for any \(n > 0\), we have \(\phi(\psi^n(E)) = \psi^n(\phi(E)) = \psi^n(E)\), so that \(\psi^n(E)\) is \(\phi\)-invariant.  But \(\phi\) can fix at most finitely divisors, and it must be that the divisors \(\psi^n(E)\) are only a finite set, so \(E\) is \(\psi\)-periodic.  After replacing \(\psi\) by an iterate, we may assume that \(E\) is invariant for both \(\phi\) and \(\psi\), and then by Lemma~\ref{autosdescend} the automorphisms \(\phi\) and \(\psi\) descend to commuting automorphisms of \(Y\).  If \(Y\) is not smooth, this establishes case (1).  If \(Y\) is smooth, we replace \(X\) with \(Y\) and continue by induction.

Suppose instead that \(E\) has infinite orbit under both \(\phi\) and \(\psi\).  Let \(\xi \subset E\) be a general \(E\)-covering curve.  By Lemma~\ref{imprimitive}, after replacing \(\phi\) and \(\psi\) by suitable iterates, both are imprimitive over the same surface \(\Hilb_{[\xi]}(X)_{\red}\):
\[
\xymatrix{
X \ar@{-->}[d] \ar[r]^{\phi,\psi} & X \ar@{-->}[d]  \\
\Hilb_{[\xi]}(X)_{\red} \ar[r]^{\bar{\phi},\bar{\psi}} & \Hilb_{[\xi]}(X)_{\red}
}
\]
If \(\bar{\phi}\) and \(\bar{\psi}\) do not coincide, then the maps \(\bar{\phi}\) and \(\bar{\psi}\) lift to commuting positive entropy automorphisms of the minimal resolution \(S\) of \(\Hilb_{[\xi]}(X)\).  By Proposition~\ref{commuteonsurface}, after replacing \(\phi\) and \(\psi\) by suitable iterates, the maps \(\bar{\phi}\) and \(\bar{\psi}\) coincide. Then \(\phi \circ \psi^{-1}\) is an automorphism of \(X\), which fixes the fibers of \(\pi : X \rat S\), and \(\phi \circ \psi^{-1} = \tau \in \Aut(X/S)\) is an automorphism of \(X\) over \(S\).  
\end{proof}

\begin{example}
Let \(S\) be a rational surface with an automorphism \(\sigma : S \to S\), and let \(\tau\) be an infinite order automorphism of \(\P^1\).  Then we can take \(\phi = \sigma \times \id\) and \(\psi = \sigma \times \tau\) to obtain two commuting automorphisms of \(X = S \times \P^1\).
\end{example}

\section{Acknowledgments}

I am indebted to a number of people for discussions of various aspects of this problem both dynamical and algebro-geometric, especially Eric Bedford, Serge Cantat, Izzet Coskun, and James M\textsuperscript{c}Kernan.  Thanks too to Burt Totaro for some useful comments.

This material is based upon work supported by the National Science Foundation under agreement No.\ DMS-1128155. Any opinions, findings and conclusions or recommendations expressed in this material are those of the author and do not necessarily reflect the views of the National Science Foundation.

\singlespacing
\bibliographystyle{amsplain}
\bibliography{zrefs}

\end{document}